\documentclass[10pt]{amsart}

\numberwithin{equation}{section}
\newcounter{mnote}
\let\oldmarginpar\marginpar
\renewcommand\marginpar[1]{\-\oldmarginpar[\raggedleft\footnotesize #1]%
	{\raggedright\footnotesize #1}}

\usepackage[english]{babel}
\usepackage{amssymb}
\usepackage{mathrsfs}
\usepackage{stmaryrd}
\usepackage{chemarrow}
\usepackage[norelsize]{algorithm2e}
\usepackage{enumerate}
\usepackage{graphicx}
\usepackage{subcaption}
\usepackage[all]{xy}
\usepackage{tikz}
\usetikzlibrary{arrows}
\usepackage{multirow}
\usepackage[colorinlistoftodos]{todonotes}
\usepackage[colorlinks=true, allcolors=blue]{hyperref}
\usepackage[hyperpageref]{backref} % 参考文献中backref
\usepackage{hyperref}
\hypersetup{hypertex=true,
	colorlinks=true,
	linkcolor=blue,
	anchorcolor=blue,
	citecolor=red}
\usepackage{verbatim}
\usepackage{booktabs}
\usepackage{xcolor}% 改变字体颜色
\allowdisplaybreaks

\newtheorem{theorem}{Theorem}[section]
\newtheorem{lemma}[theorem]{Lemma}

\newtheorem{example}[theorem]{Example}

\newtheorem{remark}[theorem]{Remark}

\newcommand{\normmm}[1]{{\left\vert\kern-0.25ex\left\vert\kern-0.25ex\left\vert #1
		\right\vert\kern-0.25ex\right\vert\kern-0.25ex\right\vert}}

\renewcommand{\div}{\operatorname{div}}
\newcommand{\grad}{\operatorname{grad}}

\newcommand{\sym}{\operatorname{sym}}
\newcommand{\skw}{\operatorname{skw}}

\begin{document}

\title[Mixed Methods for Fourth-Order Singular Problem]{Robust And Optimal Mixed Methods For A Fourth-Order Elliptic Singular Perturbation Problem}
\author{Xuehai Huang}%
\address{School of Mathematics, Shanghai University of Finance and Economics, Shanghai 200433, China}%
\email{huang.xuehai@sufe.edu.cn}%
\author{Zheqian Tang}%
\address{School of Mathematics, Shanghai University of Finance and Economics, Shanghai 200433, China}%
\email{tangzq0329@163.com}%

\thanks{%The second author is the corresponding author. 
The first author was supported by the National Natural Science Foundation of China (Grant No.\ 12171300). %, and the Natural Science Foundation of Shanghai  (Grant No.\ 21ZR1480500).
}
%%%%% Keywords %%%%%%%%%%%
%\keywords{ }
\makeatletter
\@namedef{subjclassname@2020}{\textup{2020} Mathematics Subject Classification}
\makeatother

\begin{abstract}
A series of robust and optimal mixed methods based on two mixed formulations of the fourth-order elliptic singular perturbation problem are developed in this paper. 
First, a mixed method based on a second-order system is proposed without relying on Nitsche's technique or interpolations. Robust and optimal error estimates are derived using an $L^2$-bounded interpolation operator for tensors. Then, its connections to other discrete methods, including weak Galerkin methods and a mixed finite element method based on a first-order system, are established. Finally, numerical experiments are provided to validate the theoretical results.
\end{abstract}
\maketitle

%%%%%%%%%%%%%%%%%%%%%%%%%%%%%%%%%%%%%%%%
%% sect-introduction
%%%%%%%%%%%%%%%%%%%%%%%%%%%%%%%%%%%%%%%%
\section{Introduction}
In this paper, we will propose a series of mixed methods for the fourth-order elliptic singular perturbation problem with the right-hand side $f\in L^2(\Omega)$:
\begin{equation}\label{FSP0}
\begin{cases}
\varepsilon^2\Delta^2u-\Delta u=f &\mbox{in} \ \Omega,\\
u=\partial_{n}u=0 &\mbox{on} \ \partial\Omega,
\end{cases}
\end{equation}
where $\Omega\subset\mathbb{R}^d(d\geq2)$ is a bounded polytope, 
$\partial_{n}u$ is the normal derivative of $u$, and $\varepsilon$ is a real small and positive parameter.

$H^2$-conforming elements are well-suited for discretizing both the fourth-order and second-order operators in problem \eqref{FSP0} simultaneously \cite{Semper1992,HuLinWu2024,ChenHuang2021,ChenHuang2024a,ChenChenGaoHuangEtAl2025}. However, due to the complexity of these elements, $H^2$-nonconforming elements are more widely used. Several $H^2$-nonconforming finite element methods (FEM) have been proposed in \cite{NTW2001,ChenZhaoShi2005,WangXuHu2006,WangMeng2007,ChenLiuQiao2010,XieShiLi2010,Guzman2012,WangWuXie2013,WangHuangTangZhou2018,HuangShiWang2021} to solve problem \eqref{FSP0}. In addition, a $C^0$ interior penalty discontinuous Galerkin (IPDG) method using the Lagrange element space was introduced in \cite{BrennerNeilan2011,FranzRoosWachtel2014} as an alternative approach. More recently, several virtual element methods (VEM) in \cite{ZhangZhao2020,ZhangZhao2024,FengYu2024} and a hybrid high-order method in \cite{DongErn2021} have been developed, both specifically tailored for problem \eqref{FSP0}. 

As $\varepsilon\rightarrow0$, the problem \eqref{FSP0} reduces to the Poisson equation \eqref{poisson}. While the boundary condition $\partial_{n}u=0$ may over-constrain the reduced problem, it induces boundary layer phenomena. Most of the aforementioned discrete methods are designed for the primal formulation of problem \eqref{FSP0}. However, due to the presence of boundary layers, the error estimates of these methods are typically uniform and sharp but not optimal, with a convergence rate of only half-order as $\varepsilon\rightarrow0$. 

To design robust and optimal discrete methods for the fourth-order elliptic singular perturbation problem \eqref{FSP0} in the presence of boundary layers, a key requirement is that, when $\varepsilon=0$, the scheme reduces to a standard discretization of the Poisson equation. In particular, for $\varepsilon=0$, the discrete space, bilinear form, and right-hand side should incorporate only the boundary condition $u=0$, without involving the condition $\partial_n u=0$.

One approach is to impose the boundary condition $\partial_n u=0$ weakly, for example by using Nitsche's method~\cite{Nitsche1971} or penalty techniques~\cite{ArnoldIP1982}, as in the discrete methods of~\cite{Guzman2012,HuangShiWang2021}. Another approach is to introduce an interpolation from an $H^2$ finite element space to an $H^1$ finite element space in both the right-hand side and the bilinear form associated with the Laplacian operator, as in the decoupled method of~\cite{CuiHuang2025}.

In this paper, we focus instead on more natural mixed methods that do not rely on additional stabilizations or interpolations. The mixed Hellan--Herrmann--Johnson method~\cite{Hellan1967,Herrmann1967,Johnson1973} was employed in~\cite{LiuHuangWang2020} for problem~\eqref{FSP0} in two dimensions, but without robust analysis. We first reformulate problem~\eqref{FSP0} as a second-order system
\begin{equation*}
\begin{cases}
\varepsilon^{-2}\boldsymbol{\sigma} = \nabla^2 u &\mbox{in} \ \Omega,\\
\div\div\boldsymbol{\sigma} - \Delta u = f &\mbox{in} \ \Omega,\\
u=\partial_{n}u=0 &\mbox{on} \ \partial\Omega.
\end{cases}
\end{equation*}
A distributional mixed formulation of this second-order system is to find $(\boldsymbol{\sigma},u)\in H^{-1}(\div\div,\Omega; \mathbb{M})\times H^1_0(\Omega)$ such that 
\begin{subequations}\label{intro-mix}
\begin{align}\label{intro-mix1}
\varepsilon^{-2}(\boldsymbol{\sigma},\boldsymbol{\tau})-\langle\div\div\boldsymbol{\tau}, u\rangle &= 0 \quad\quad\quad\;\; \forall\, \boldsymbol{\tau}\in  H^{-1}(\div\div,\Omega; \mathbb{M}), \\
\label{intro-mix2}
-\langle\div\div\boldsymbol{\sigma}, v\rangle-(\nabla u,\nabla v) &= -(f,v) \quad \forall\, v\in H^1_0(\Omega),
\end{align}
\end{subequations}
where the Hilbert space
$$H^{-1}(\div\div,\Omega;\mathbb{M}):=\{\boldsymbol{\tau}\in L^2(\Omega;\mathbb{M}): \div\div\boldsymbol{\tau}\in H^{-1}(\Omega)\}$$
with $\mathbb{M}:=\mathbb R^{d\times d}$. It is straightforward to observe that the boundary condition $\partial_{n}u=0$ has been weakly imposed in \eqref{intro-mix1}. We then proceed to develop the corresponding numerical scheme based on the formulation \eqref{intro-mix}. When $\boldsymbol{\tau}\in H(\div,\Omega; \mathbb{M})\subset H^{-1}(\div\div,\Omega; \mathbb{M})$, we have 
$$
-\langle\div\div\boldsymbol{\tau}, v\rangle=(\div\boldsymbol{\tau}, \nabla v)\qquad\forall\,\boldsymbol{\tau}\in H(\div,\Omega; \mathbb{M}), v\in H_0^1(\Omega),
$$
where the Hilbert space
$$
H(\div,\Omega; \mathbb{M}):=\{\boldsymbol{\tau}\in L^2(\Omega;\mathbb{M}): \div\boldsymbol{\tau}\in L^{2}(\Omega;\mathbb R^d)\}.
$$

We discretize $\boldsymbol{\sigma}$ using $H(\div)$-conforming finite elements for tensors with  the shape function space
$$
\Sigma_{r,k,m}(T;\mathbb{M}):=\mathbb{P}_r(T;\mathbb{M})+\mathbb{P}_{k-1}(T;\mathbb{R}^d)\otimes \boldsymbol{x} + (\boldsymbol{x}\otimes\boldsymbol{x})\mathbb{H}_{m-2}(T),
$$
where integers $k\geq1$, $r\geq0$ and $m\geq1$ satisfying the constraint
\begin{equation}\label{intro:conditionkrm}
r=k, k-1; \quad m=k,k+1;\quad r\leq m\leq r+1.
\end{equation}
It is the tensor version of the Brezzi--Douglas--Marini (BDM) element \cite{BrezziDouglasMarini1985,BrezziDouglasDuranFortin1987,nedelec1986new} for $r=m=k$, and the Raviart--Thomas (RT) element \cite{RaviartThomas1977,Nedelec1980} for $r=k-1$ and $m=k$. This tensor-valued finite element is new for $r=k$ and $m=k+1$. The global finite element space for tensors is defined as
\begin{equation*}
\Sigma_{r,k,m,h}^{\div}:=\{\boldsymbol{\tau}\in H(\div,\Omega;\mathbb{M}):\boldsymbol{\tau}|_T\in \Sigma_{r,k,m}(T;\mathbb{M})~\textrm{for}~\textrm{each}~T\in\mathcal{T}_h\}. 
\end{equation*}
For the discretization of $u$, we consider the $H^1$-nonconforming space $\mathring{V}_{k,m,h}^{\rm VE}$ (cf. \eqref{eq:VEMspace}) with the following local degrees of freedoms (DoFs)
\begin{subequations}\label{intro-vedof}
\begin{align}
\label{intro-vedof1}
\dfrac{1}{|F|}(v,q)_F, & \quad q\in \mathbb{P}_{k-1}(F), \, F\in\mathcal{F}_h,\\
\label{intro-vedof2}
\dfrac{1}{|T|}(v,q)_T, & \quad q\in \mathbb{P}_{m-2}(T), \, T\in\mathcal{T}_h,
\end{align}
\end{subequations}
where $k\geq1$ and $m=k,k+1$.
For $k=1$, $\mathring{V}_{k,m,h}^{\rm VE}$ is the $H^1$-nonconforming finite element space in \cite{CR1973,HuHuangLin2014,HuMa2015}. When $k\geq2$, $\mathring{V}_{k,m,h}^{\rm VE}$ is the $H^1$-nonconforming virtual element space on simplicial meshes in \cite{ncvem2016,ChenHuangncvem2020}.
We establish the following orthogonality property for all integers $k\geq1$, $r\geq0$ and $m\geq1$ satisfying the constraint \eqref{intro:conditionkrm}, i.e., \eqref{exchange}:
\begin{equation}\label{intro-int-eq}
(\div\boldsymbol{\tau},\nabla(I^{\rm VE}_{h}u-u))_T=0  \quad \forall\, \boldsymbol{\tau}\in \Sigma_{r,k,m}(T;\mathbb{M}), T\in\mathcal{T}_h,
\end{equation}
where $I^{\rm VE}_{h}: H_0^1(\Omega)\rightarrow \mathring{V}_{k,m,h}^{\rm VE}$ is the canonical interpolation operator based on DoFs \eqref{intro-vedof}.
Notably, if $u$ is discretized using the $k$th order Lagrange element, the corresponding interpolation operator no longer satisfies \eqref{intro-int-eq}, leading to a suboptimal discrete method.

An additional stabilization term is typically required in the VEMs to ensure the coercivity of the discrete bilinear form for the Poisson equation in \cite{BeiraoBrezzi2013,BeiraoBrezzi2014}. However, as in \cite{ChenHuangWei2024vem}, we prove the norm equivalence
\begin{equation*}
\|Q_h^{\rm div}\nabla_h v_h\|_0\eqsim\|\nabla_h v_h\|_{0} \quad \forall\, v_h\in \mathring{V}_{k,m,h}^{\rm VE},
\end{equation*} 
where $Q_h^{\rm div}$ denotes the $L^2$ projector onto the piecewise BDM or RT element space (cf. Section~\ref{sec:divfem}).
This allows us to construct a mixed method without extrinsic stabilization. By utilizing an $L^2$-bounded interpolation operator for tensors and interpolation error estimates, we derive optimal and robust error estimates for the proposed mixed method, independent of the presence of boundary layers. 

Next, we connect the proposed mixed method to a stabilization-free weak Galerkin method by reformulating it using the weak gradient operator $\nabla_w$ and the isomorphism between the virtual element space and the Lagrange multiplier space.
Furthermore, by applying the hybridization technique \cite{ArnoldBrezzi1985} to relax the normal continuity of the finite element space for tensors, we achieve a fully weak Galerkin formulation.

Although the mixed method developed from \eqref{intro-mix} employs virtual elements, it remains equivalent to a mixed FEM. We recast problem \eqref{FSP0} as the following first-order system
\begin{equation*}
\begin{cases}
\boldsymbol{p}=\nabla u,\; \varepsilon^{-2}\boldsymbol{\sigma} = \nabla \boldsymbol{p},  &\mbox{in} \ \Omega,\\
\boldsymbol{\phi}=\div\boldsymbol{\sigma} - \boldsymbol{p}, \; \div\boldsymbol{\phi} = f &\mbox{in} \ \Omega,\\
u=0, \; \boldsymbol{p}=0 &\mbox{on} \ \partial\Omega.
\end{cases}
\end{equation*}
Then we propose the mixed FEM \eqref{FSP1st-mfem} based on this first-order system to approximate $(\boldsymbol{\sigma},\boldsymbol{\phi},\boldsymbol{p}, u)$ simultaneously. We find that when $(\boldsymbol{\sigma}_h,u_h)\in\Sigma_{r,k,m,h}^{\div}\times\mathring{V}_{k,m,h}^{\rm VE}$ is the solution of the mixed method based on \eqref{intro-mix}, the tuple $(\boldsymbol{\sigma}_h,\div\boldsymbol{\sigma}_h-Q_h^{\rm div}\nabla_h u_h,Q_h^{\rm div}\nabla_h u_h, Q_{m-2,h}u_h)$ coincides with the solution of the mixed FEM \eqref{FSP1st-mfem}, where $Q_{m-2,h}$ denotes the projector onto piecewise $(m-2)$th order polynomial space. This means the two mixed methods are equivalent.

The rest of this paper is organized as follows. Two mixed formulations for the fourth-order elliptic singular perturbation problem are presented in Section \ref{sec2}. Section \ref{sec3} introduces some discrete spaces and interpolation results. In Section \ref{sec4}, we propose a mixed method and provide its corresponding error analysis. Section~\ref{sec5} connects the mixed method with other methods, and the numerical experiments in the last section validate the theoretical results.

%%%%%%%%%%%%%%%%%%%%%%%%%%%%%%%%%%%%%%%%
%% sect-Preliminaries
%%%%%%%%%%%%%%%%%%%%%%%%%%%%%%%%%%%%%%%%
\section{Mixed Formulations}\label{sec2}

Two mixed formulations for the fourth-order elliptic singular perturbation problem \eqref{FSP0} are shown in this section, namely a distributional mixed formulation based on the second-order system \eqref{FSP1} and a mixed formulation based on the first-order system \eqref{FSP11}.

\subsection{Notation}
Let $\Omega\subset\mathbb{R}^d$ ($d\geq 2$) be a bounded polytope with boundary $\partial\Omega$.
Given a bounded domain $D$ and a real number $m$, denote by $H^m(D) $ the standard Sobolev space on $D$ with norm $\|\cdot\|_{m,D}$ and semi-norm $|\cdot |_{m, D}$, and $H_0^m( D)$ the closure of $C_0^\infty(D)$ with respect to $\|\cdot\|_{m,D}$. The notation $(\cdot,\cdot)_D$ symbolizes the $L^2$ inner product on $D$. Refer to \cite{GiraultRaviart1986}, the Sobolev space $H(\div, D)$ with norm $\|\cdot\|_{H(\div,D)}$ is defined in the standard way. For $D= \Omega$, we abbreviate $\|\cdot \|_{m, D}$, $|\cdot |_{m, D}$, $( \cdot , \cdot ) _D$ and $\|\cdot\|_{H(\div,D)}$ as $\|\cdot \|_m, |\cdot |_m$, $(\cdot,\cdot)$ and $\|\cdot\|_{H(\div)}$, respectively. The duality pairing between a Banach space $V$ and its dual $V'$ is denoted by $\langle \cdot, \cdot \rangle_{V' \times V}$, which will be abbreviated as $\langle \cdot, \cdot \rangle$ whenever no ambiguity arises. Denote by $h_D$ the diameter of $D$.

For a $d$-dimensional simplex $T$, let $\mathcal{F}(T)$ denote the set of all $(d-1)$-dimensional faces of $T$. We use $\boldsymbol{n}_{\partial T}$ to denote the unit outward normal vector of $\partial T$, which is a piecewise constant vector function.
For each face $F \in \mathcal{F}(T)$, we fix a unit normal vector $\boldsymbol{n}_F$. We will abbreviate $\boldsymbol{n}_{\partial T}$ and $\boldsymbol{n}_F$ as $\boldsymbol{n}$ if not causing any confusion. 

Given a face $F\in\mathcal{F}(T)$ and a vector $\boldsymbol{v}\in\mathbb{R}^d$, define its projection to the plane $F$
\begin{equation*}
\Pi_F\boldsymbol{v} = (\boldsymbol{I}-\boldsymbol{n}_F\boldsymbol{n}^\intercal_F)\boldsymbol{v},
\end{equation*}
which is called the tangential component of $\boldsymbol{v}$. 
For a scalar function $v$, let $\nabla v$ and $\grad v$ be the gradient of $v$, which is treated as a column vector.
Define the surface gradient
\begin{equation*}
\nabla_Fv=\Pi_F(\nabla v) = \nabla v-\dfrac{\partial v}{\partial n_F}\boldsymbol{n}_F,
\end{equation*}
namely the projection of $\nabla v$ to the face $F$, which is independent of the choice of the normal vector. Let $\div\boldsymbol{v}$ be the divergence of a vector function  $\boldsymbol{v}$.

Denote by $\mathcal{T}_h=\{T\}$ a conforming simplicial mesh of $\Omega$ with each element being a simplex, where $h:=\max_{T\in\mathcal{T}_h}h_T$ and $h_T=\mathrm{diam}(T)$.
Let $\mathcal{F}_h$ and $\mathring{\mathcal{F}}_h$ be the set of all $(d-1)$-dimensional faces and interior $(d-1)$-dimensional faces, respectively.
Consider two adjacent simplices $T^+$ and $T^-$ sharing an interior face $F$. Define the jump of a function $v$ on $F$ as
\begin{align*}
[\![v]\!]|_F:=(v|_{T^+})|_F\boldsymbol{n}_F\cdot\boldsymbol{n}_{\partial T^+}+(v|_{T^-})|_F\boldsymbol{n}_F\cdot\boldsymbol{n}_{\partial T^-}.
\end{align*}
On a face $F$ lying on the boundary $\partial\Omega$, the above term becomes $[\![v]\!]|_F=v|_F$.

For a bounded domain $D\subset \mathbb{R}^d$ and a non-negative integer $k$, let $\mathbb{P}_k(D)$ stand for the set of all polynomials over $D$ with the total degree no more than $k$. Set $\mathbb{P}_{k}(D) = \{ 0\}$ for $k<0$. Let $\mathbb{H}_k(D):= \mathbb{P}_{k}(D)\backslash\mathbb{P}_{k-1}(D)$ be the space of homogeneous polynomials of degree $k$. Denote by $Q_{k,D}$ the standard $L^2$-projection operator from $L^2(D)$ to $\mathbb{P}_k(D)$, whose vectorial/tensorial version is also denoted by $Q_{k,D}$ if there is no confusion. Let $Q_{k,h}$ be the element-wise version of $Q_{k,D}$ with respect to $\mathcal{T}_h$.
For $s\geq 1$ and integer $k\geq0$, introduce
\begin{align*}
H^s(\mathcal{T}_h)&:=\{v\in L^2(\Omega): v|_T\in H^s(T) \quad \forall\, T\in\mathcal{T}_h\}, \\
\mathbb P_k(\mathcal{T}_h)&:=\{v\in L^2(\Omega): v|_T\in \mathbb P_k(T) \quad \forall\, T\in\mathcal{T}_h\}, \\
\mathbb P_k(\mathcal{F}_h)&:=\{v\in L^2(\mathcal{F}_h): v|_F\in \mathbb P_k(F) \quad \forall\, F\in\mathcal{F}_h\}.	
\end{align*}
Denote $\mathbb P_k(\mathring{\mathcal{F}_h}):=\mathbb P_k(\mathcal{F}_h)\cap L^2(\mathring{\mathcal{F}_h})$, where
\begin{equation*}
L^2(\mathring{\mathcal{F}_h}):=\{v\in L^2(\mathcal{F}_h): v|_F=0 \quad \forall\, F\in\mathcal{F}_h\backslash\mathring{\mathcal{F}_h}\}.
\end{equation*}

Set $\mathbb{M}:= \mathbb{R}^{d\times d}$. Denote by $\mathbb{S}$ and $\mathbb{K}$ the subspace of symmetric matrices and skew-symmetric matrices of $\mathbb{M}$, respectively.
Each tensor $\boldsymbol{\tau}\in\mathbb{M}$ can be decomposed as $\boldsymbol{\tau} = \sym\boldsymbol{\tau}+\skw\boldsymbol{\tau}$, where the symmetric part $\sym\boldsymbol{\tau}\in\mathbb{S}$ and the skew-symmetric part $\skw\boldsymbol{\tau}\in \mathbb{K}$ are defined as
$$
\sym\boldsymbol{\tau}:=\frac{1}{2}(\boldsymbol{\tau}+\boldsymbol{\tau}^\intercal),\quad \skw\boldsymbol{\tau}:=\frac{1}{2}(\boldsymbol{\tau}-\boldsymbol{\tau}^\intercal).
$$
In addition, for a space $\mathbb{B}(D)$ defined on $D$, let $\mathbb{B}(D;\mathbb{X}):=\mathbb{B}(D)\otimes\mathbb{X}$ be its vector or tensor version for $\mathbb{X}$ being $\mathbb{R}^{d}, \mathbb{M}, \mathbb{S}$ and $\mathbb{K}$. 

For a vector-valued function $\boldsymbol{v}$, define 
$$
\nabla\boldsymbol{v}:=\nabla\otimes\boldsymbol{v}, \quad
\grad\boldsymbol{v}:=(\nabla\boldsymbol{v})^{\intercal}=\boldsymbol{v}\otimes\nabla.
$$
For a tensor-valued function $\boldsymbol{\tau}=(\tau_{ij})_{i,j=1}^{d\times d}$, denote by $\div\boldsymbol{\tau}$ the row-wise divergence of $\boldsymbol{\tau}$, i.e., $(\div\boldsymbol{\tau})_{i}=\sum_{j=1}^d\partial_j\tau_{ij}$ for $i=1,\dots,d$.
We use $\nabla_h$ and $\div_h$ to represent the element-wise gradient and $\div$ with respect to $\mathcal{T}_h$. 
For a piecewise smooth scalar, vector-valued or tensor-valued function $v$, define the broken squared seminorm with $s\geq 1$
$$|v|_{s,h}^2:=\sum_{T\in\mathcal{T}_h}|v|^2_{s,T}.$$
In this paper, we use $\lesssim$ to represent $\leq C$, where $C$ is a generic positive constant independent of the mesh size $h$ and the parameter $\varepsilon$. And $a\eqsim b$ means $a\lesssim b\lesssim a$.

\subsection{A distributional mixed formulation}
% In this subsection, assume $f\in H^{-1}(\Omega)$. 
Introducing $\boldsymbol{\sigma} := \varepsilon^2\nabla^2 u$, rewrite the fourth-order elliptic singular perturbation problem \eqref{FSP0} as the following second-order system
\begin{equation}\label{FSP1}
\begin{cases}
\varepsilon^{-2}\boldsymbol{\sigma} = \nabla^2 u &\mbox{in} \ \Omega,\\
\div\div\boldsymbol{\sigma} - \Delta u = f &\mbox{in} \ \Omega,\\
u=\partial_{n}u=0 &\mbox{on} \ \partial\Omega.
\end{cases}
\end{equation}
Define the Hilbert space
$$H^{-1}(\div\div,\Omega;\mathbb{M}):=\{\boldsymbol{\tau}\in L^2(\Omega;\mathbb{M}): \div\div\boldsymbol{\tau}\in H^{-1}(\Omega)\}$$
with squared norm 
$$
\|\boldsymbol{\tau}\|^2_{\varepsilon^{-1},H^{-1}(\div\div)}:=\varepsilon^{-2}\|\boldsymbol{\tau}\|_0^2+\|\div\div\boldsymbol{\tau}\|_{-1}^2, 
$$
where
$$\|\div\div\boldsymbol{\tau}\|_{-1}: = \sup_{v\in H^1_0(\Omega),v\neq0}\dfrac{\langle\div\div\boldsymbol{\tau}, v\rangle}{|v|_1}.$$
A distributional mixed formulation of the second-order system \eqref{FSP1} is to find $(\boldsymbol{\sigma},u)\in H^{-1}(\div\div,\Omega; \mathbb{M})\times H^1_0(\Omega)$ such that
\begin{subequations}\label{FSP-mix}
\begin{align}
\label{FSP-mix1}
a(\boldsymbol{\sigma},\boldsymbol{\tau})+b(\boldsymbol{\tau},u) &= 0 \quad\quad\quad\;\; \forall\, \boldsymbol{\tau}\in  H^{-1}(\div\div,\Omega; \mathbb{M}), \\
\label{FSP-mix2}
b(\boldsymbol{\sigma},v)-c(u,v) &= -(f,v) \quad \forall\, v\in H^1_0(\Omega),
\end{align}
\end{subequations}
where 
$$a(\boldsymbol{\sigma},\boldsymbol{\tau}):=\varepsilon^{-2}(\boldsymbol{\sigma},\boldsymbol{\tau}), \quad b(\boldsymbol{\tau},v):=-\langle\div\div\boldsymbol{\tau}, v\rangle, \quad c(u,v):=(\nabla u,\nabla v).$$

\begin{theorem}
The distributional mixed formulation \eqref{FSP-mix} is well-posed with the solution $(\boldsymbol{\sigma},u)\in H^{-1}(\div\div,\Omega; \mathbb{M})\times H^1_0(\Omega)$. Furthermore, $\boldsymbol{\sigma}\in H^{-1}(\div\div,\Omega; \mathbb{S})$, where
$$H^{-1}(\div\div,\Omega;\mathbb{S}):=\{\boldsymbol{\tau}\in L^2(\Omega;\mathbb{S}): \div\div\boldsymbol{\tau}\in H^{-1}(\Omega)\}.$$
Notice that $H^{-1}(\Omega)$ is the dual space of $H_0^1(\Omega)$.
\end{theorem}
\begin{proof}
It can be readily verified that 
\begin{align*}
a(\boldsymbol{\tau}, \boldsymbol{\tau})&+ \sup_{v\in H^1_0(\Omega)}\dfrac{b^2(\boldsymbol{\tau}, v)}{|v|_1^2} =\|\boldsymbol{\tau}\|^2_{\varepsilon^{-1},H^{-1}(\div\div)} \quad \forall\, \boldsymbol{\tau}\in H^{-1}(\div\div,\Omega; \mathbb{M}),\\		
c(v, v) &+ \sup_{\boldsymbol{\tau}\in H^{-1}(\div\div,\Omega; \mathbb{M})}\dfrac{b^2(\boldsymbol{\tau}, v)}{\|\boldsymbol{\tau}\|_{\varepsilon^{-1},H^{-1}(\div\div)}^2} \eqsim |v|^2_1 \quad \forall\, v\in H^1_0(\Omega).
\end{align*}
Then by the Zulehner theory \cite[Theorem 2.6]{zulehner2011}, the mixed formulation \eqref{FSP-mix} is well-posed.

Take $\boldsymbol{\tau}=\skw\boldsymbol{\sigma}\in H^{-1}(\div\div,\Omega; \mathbb{M})$ in equation \eqref{FSP-mix1}. By $\div\div\boldsymbol{\tau}=0$, we have $\|\skw\boldsymbol{\sigma}\|_0^2=0$. Hence $\boldsymbol{\sigma}\in H^{-1}(\div\div,\Omega; \mathbb{S})$. 
\end{proof}

Although $\boldsymbol{\sigma}\in H^{-1}(\div\div,\Omega; \mathbb{M})$ explicitly, equation \eqref{FSP-mix1} inherently indicates that $\boldsymbol{\sigma}\in H^{-1}(\div\div,\Omega; \mathbb{S})$.
When $\boldsymbol{\tau}\in H(\div,\Omega; \mathbb{M})\subset H^{-1}(\div\div,\Omega; \mathbb{M})$, we have 
$$
-\langle\div\div\boldsymbol{\tau}, v\rangle=(\div\boldsymbol{\tau}, \nabla v)\qquad\forall\,\boldsymbol{\tau}\in H(\div,\Omega; \mathbb{M}), v\in H_0^1(\Omega).
$$
We will discretize $\boldsymbol{\sigma}$ using a finite element subspace of $H(\div,\Omega; \mathbb{M})$ in later sections.	

\begin{remark}\rm
Unlike the distributional mixed formulation in \cite{ChenHuang2018,LiuHuangWang2020,ChenHuHuang2018,HHJ2016,ChenHuang2025,Hellan1967,Herrmann1967,Johnson1973}, $\boldsymbol{\sigma}$ here is chosen to belong to $H^{-1}(\div\div,\Omega; \mathbb{M})$ instead of $H^{-1}(\div\div,\Omega; \mathbb{S})$.
\end{remark}

\begin{lemma}
The problem \eqref{FSP0} and the mixed formulation \eqref{FSP-mix} are equivalent.
\end{lemma}
\begin{proof}
Since both problems are uniquely solvable, it suffices to show that $(\boldsymbol{\sigma},u)$ with $\boldsymbol{\sigma} = \varepsilon^2\nabla^2u$ solves the mixed formulation \eqref{FSP-mix}, if $u$ solves problem \eqref{FSP0}. Assume $ u \in H_0^2(\Omega) $ is the solution of problem \eqref{FSP0}. Then $\boldsymbol{\sigma} \in H^{-1}(\div\div, \Omega; \mathbb{M})$, and equation \eqref{FSP-mix2} is derived from \eqref{FSP0} by applying integration by parts to $-(\Delta u,v)$.  
Using the definition of $\div\div \boldsymbol{\tau}$ in the distributional sense and the density of $ C_0^\infty(\Omega) $ in $ H_0^1(\Omega)$, equation \eqref{FSP-mix1} follows from $\boldsymbol{\sigma} = \varepsilon^2 \nabla^2 u$.
\end{proof}

\subsection{A mixed formulation based on the first-order system}
Rewrite the fourth-order elliptic singular perturbation problem \eqref{FSP0} as the following first-order system
\begin{equation}\label{FSP11}
\begin{cases}
\boldsymbol{p}=\nabla u,\; \varepsilon^{-2}\boldsymbol{\sigma} = \nabla \boldsymbol{p},  &\mbox{in} \ \Omega,\\
\boldsymbol{\phi}=\div\boldsymbol{\sigma} - \boldsymbol{p}, \; \div\boldsymbol{\phi} = f &\mbox{in} \ \Omega,\\
u=0, \; \boldsymbol{p}=0 &\mbox{on} \ \partial\Omega.
\end{cases}
\end{equation}
A mixed formulation of this first-order system is to find $(\boldsymbol{\sigma},\boldsymbol{\phi},\boldsymbol{p}, u)\in H(\div,\Omega; \mathbb{M})\times H(\div,\Omega)\times L^2(\Omega;\mathbb R^d)\times L^2(\Omega)$ such that
\begin{subequations}\label{FSP1st-mix}
\begin{align}
\label{FSP1st-mix1}
\varepsilon^{-2}(\boldsymbol{\sigma},\boldsymbol{\tau})+\bar{b}(\boldsymbol{\tau},\boldsymbol{\psi};\boldsymbol{p}, u) &= 0 \quad\quad\quad\;\; \forall\, \boldsymbol{\tau}\in  H(\div,\Omega; \mathbb{M}), \boldsymbol{\psi}\in  H(\div,\Omega), \\
\label{FSP1st-mix2}
\bar{b}(\boldsymbol{\sigma},\boldsymbol{\phi};\boldsymbol{q}, v)-(\boldsymbol{p},\boldsymbol{q}) &= -(f,v) \quad \forall\, \boldsymbol{q}\in L^2(\Omega;\mathbb R^d), v\in L^2(\Omega),
\end{align}
\end{subequations}
where 
$$
\bar{b}(\boldsymbol{\tau},\boldsymbol{\psi};\boldsymbol{q}, v):=(\div\boldsymbol{\tau}-\boldsymbol{\psi}, \boldsymbol{q})-(\div\boldsymbol{\psi}, v).
$$

For $(\boldsymbol{\tau}, \boldsymbol{\psi})\in H(\div,\Omega; \mathbb{M})\times H(\div,\Omega)$, we equip the following parameter-dependent norm
$$
\|(\boldsymbol{\tau}, \boldsymbol{\psi})\|^2_{\varepsilon^{-1},\div}:=\varepsilon^{-2}\|\boldsymbol{\tau}\|_0^2+\|\div\boldsymbol{\tau}-\boldsymbol{\psi}\|_{0}^2+\|\div\boldsymbol{\psi}\|_{0}^2. 
$$

\begin{theorem}
The mixed formulation \eqref{FSP1st-mix} is well-posed.
\end{theorem}
\begin{proof}
It can be readily verified that 
\begin{equation*}
\varepsilon^{-2}\|\boldsymbol{\tau}\|_0^2+ \sup_{\boldsymbol{q}\in L^2(\Omega;\mathbb R^d), v\in L^2(\Omega)}\frac{\bar{b}^2(\boldsymbol{\tau},\boldsymbol{\psi};\boldsymbol{q}, v)}{\|\boldsymbol{q}\|_0^2+\|v\|_0^2} \eqsim \|(\boldsymbol{\tau}, \boldsymbol{\psi})\|^2_{\varepsilon^{-1},\div}
\end{equation*}
for $\boldsymbol{\tau}\in  H(\div,\Omega; \mathbb{M})$ and $\boldsymbol{\psi}\in  H(\div,\Omega)$, and 
\begin{equation*}
\|\boldsymbol{q}\|_0^2 + \beta^2 \lesssim  \|\boldsymbol{q}\|_0^2+\|v\|_0^2\;\; \textrm{ with }\; \beta:=\sup_{\boldsymbol{\tau}\in  H(\div,\Omega; \mathbb{M}), \boldsymbol{\psi}\in  H(\div,\Omega)}\frac{\bar{b}(\boldsymbol{\tau},\boldsymbol{\psi};\boldsymbol{q}, v)}{\|(\boldsymbol{\tau}, \boldsymbol{\psi})\|_{\varepsilon^{-1},\div}}
\end{equation*}
for $\boldsymbol{q}\in L^2(\Omega;\mathbb R^d)$ and $v\in L^2(\Omega)$.
Thanks to $\div H(\div,\Omega)=L^2(\Omega)$,
\begin{equation*}
\|v\|_0\lesssim \sup_{\boldsymbol{\psi}\in  H(\div,\Omega)}\frac{-(\div\boldsymbol{\psi}, v)}{\|\boldsymbol{\psi}\|_{H(\div)}}=\sup_{\boldsymbol{\psi}\in  H(\div,\Omega)}\frac{\bar{b}(0,\boldsymbol{\psi};\boldsymbol{q}, v)+(\boldsymbol{\psi},\boldsymbol{q})}{\|\boldsymbol{\psi}\|_{H(\div)}} \leq \beta + \|\boldsymbol{q}\|_0.
\end{equation*}
Combining the last two inequalities gives
\begin{equation*}
\|\boldsymbol{q}\|_0^2 + \beta^2 \eqsim  \|\boldsymbol{q}\|_0^2+\|v\|_0^2 \qquad \forall\,\boldsymbol{q}\in L^2(\Omega;\mathbb R^d), v\in L^2(\Omega).
\end{equation*}
Then by the Zulehner theory \cite[Theorem 2.6]{zulehner2011}, the mixed formulation \eqref{FSP1st-mix} is well-posed.
\end{proof}

By \eqref{FSP1st-mix1}, we have $\boldsymbol{\sigma}=\varepsilon^{2}\nabla^2u$, thus the solution $\boldsymbol{\sigma}$ of the mixed formulation~\eqref{FSP1st-mix} is symmetric.

\begin{remark}\rm
By $\boldsymbol{\sigma}=\varepsilon^{2}\nabla^2u$ and $\boldsymbol{\sigma}\in H(\div,\Omega; \mathbb{M})$, the solution $u\in H_0^2(\Omega)$ of the mixed formulation~\eqref{FSP1st-mix} satisfies $\Delta u\in H^1(\Omega)$.
If the solution $u \in H_0^2(\Omega)$ of the original problem~\eqref{FSP0} also satisfies $\Delta u \in H^1(\Omega)$, then the mixed formulation~\eqref{FSP1st-mix} is equivalent to~\eqref{FSP0}.
However, if $\Delta u \notin H^1(\Omega)$, the two formulations are not necessarily equivalent.
For related results on the biharmonic equation, we refer to~\cite{NazarovSweers2007,ZhangZhang2008,GerasimovStylianouSweers2012}.
\end{remark}

\subsection{Regularity}
Taking $\varepsilon = 0$, problem \eqref{FSP0} becomes the Poisson equation
\begin{equation}\label{poisson}
\begin{cases}
-\Delta \bar{u}=f &\mbox{in} \ \Omega,\\
\bar{u}=0 &\mbox{on} \ \partial\Omega.
\end{cases}
\end{equation}
We assume the Poisson equation \eqref{poisson} has the $s$-regularity with $s\geq 2$ 
\begin{equation}\label{poissonregularity}
\|\bar{u}\|_s\lesssim\|f\|_{s-2}.
\end{equation}
If $\Omega$ is semi-convex (i.e., a Lipschitz domain satisfying a uniform exterior ball condition; cf. \cite{MitreaMitreaYan2010}) or if the closure of $\Omega$ has a uniformly positive reach (cf. \cite[Definition 1.2]{GaoLai2020}), the regularity result \eqref{poissonregularity} for $s = 2$ can be found in \cite{GaoLai2020,MitreaMitreaYan2010,Adolfsson1992,Kadlec1964,Talenti1965}. Notably, any convex domain is also semi-convex.

Assume the fourth-order elliptic singular perturbation problem \eqref{FSP0} possesses
the following regularity 
\begin{equation}\label{regularity}
|u-\bar{u}|_1+\varepsilon|u|_2+\varepsilon^2|u|_3\lesssim\varepsilon^{1/2}\|f\|_0.
\end{equation}
The regularity \eqref{regularity} holds when $\Omega$ is convex in two and three dimensions;
see \cite[Lemma 5.1]{NTW2001} and \cite[Lemma 4]{Guzman2012}.

%%%%%%%%%%%%%%%%%%%%%%%%%%%%%%%%%%%%%%%%
%% Technical tools
%%%%%%%%%%%%%%%%%%%%%%%%%%%%%%%%%%%%%%%%
\section{Discrete Spaces and Interpolations}\label{sec3}
In this section, we present $H(\div)$-conforming finite elements, an $H^1$-nonconforming virtual element, and their corresponding interpolation operators.

\subsection{$H(\div)$-conforming finite elements}\label{sec:divfem}
Recall the Brezzi--Douglas--Marini (BDM) element \cite{ChenHuang2022div,ChenChenHuangWei2024,ChenHuang2024,BrezziDouglasMarini1985,BrezziDouglasDuranFortin1987,nedelec1986new} and the Raviart--Thomas (RT) element \cite{RaviartThomas1977,Nedelec1980}.
Let $T$ be a $d$-dimensional simplex. For integers $k\geq1$ and $m\geq2$ with $m=k,k+1$,
take $V_{k-1,m-1}(T;\mathbb R^d):=\mathbb{P}_{k-1}(T;\mathbb{R}^d)+\boldsymbol{x}\mathbb{H}_{m-2}(T)$ as the space of shape functions, and the degrees of freedom (DoFs) are given by (cf. \cite{ChenHuang2022div})
\begin{subequations}\label{div-dof}
\begin{align}
\label{div-dof1}
(\boldsymbol{v}\cdot \boldsymbol n, q)_F, & \quad q\in \mathbb{P}_{k-1}(F), \, F\in\mathcal{F}(T),\\
\label{div-dof2}
(\boldsymbol{v}, \boldsymbol{q})_T, & \quad  \boldsymbol{q}\in 
\nabla\mathbb{P}_{m-2}(T)\oplus (\mathbb{P}_{k-2}(T;\mathbb{R}^d)\cap\ker(\cdot\boldsymbol{x})),
\end{align}
where $\mathbb{P}_{k-2}(T;\mathbb{R}^d)\cap\ker(\cdot\boldsymbol{x}):=\{\boldsymbol{q}\in\mathbb{P}_{k-2}(T;\mathbb{R}^d): \boldsymbol{q}\cdot\boldsymbol{x}=0\}$.
\end{subequations}
By Lemma 3.9 in~\cite{ChenHuang2022div},
\begin{align}
\label{eq:divDoFintfullk}
\mathbb{P}_{m-3}(T;\mathbb R^d) & = \nabla\mathbb{P}_{m-2}(T)\oplus (\mathbb{P}_{m-3}(T;\mathbb{R}^d)\cap\ker(\cdot\boldsymbol{x})) \\
\notag
& \subseteq \nabla\mathbb{P}_{m-2}(T)\oplus (\mathbb{P}_{k-2}(T;\mathbb{R}^d)\cap\ker(\cdot\boldsymbol{x})).
\end{align}
Set $V_{0,0}(T;\mathbb{R}^d) := \mathbb{P}_0(T;\mathbb{R}^d)$, corresponding to the case $k=1$ and $m=1$, which is not uniquely determined by the DoFs \eqref{div-dof}.
For $k\geq1$ and $m\geq1$, let $Q_T^{\rm div}: L^2(T; \mathbb{R}^d) \to V_{k-1,m-1}(T;\mathbb R^d)$ be the $L^2$-orthogonal projection operator, and let $Q_h^{\rm div}$ denote its element-wise version with respect to the mesh $\mathcal{T}_h$. As $\mathbb{P}_{k-1}(T;\mathbb{R}^d)\subseteq V_{k-1,m-1}(T;\mathbb R^d)$, we have for any $1\leq s\leq k$ and $T\in\mathcal T_h$ that
\begin{equation}\label{estimateQTdiv}
\|\boldsymbol{v}-Q_h^{\rm div}\boldsymbol{v}\|_{0,T} + h_T|\boldsymbol{v}-Q_h^{\rm div}\boldsymbol{v}|_{1,T} \lesssim h_T^s|\boldsymbol{v}|_{s,T}\quad\forall~\boldsymbol{v}\in H^s(\mathcal{T}_h;\mathbb{R}^d).
\end{equation}

We will adopt tensor-valued $H(\div)$-conforming finite elements for discretizing $\boldsymbol{\sigma}\in H^{-1}(\div\div,\Omega; \mathbb{M})$.
Hereafter, we always assume integers $k\geq1$, $r\geq0$ and $m\geq1$ satisfying the constraint
\begin{equation*}
r=k, k-1; \quad m=k,k+1;\quad r\leq m\leq r+1.
\end{equation*}
Define the shape function space as
$$
\Sigma_{r,k,m}(T;\mathbb{M}):=\mathbb{P}_r(T;\mathbb{M})+\mathbb{P}_{k-1}(T;\mathbb{R}^d)\otimes \boldsymbol{x} + (\boldsymbol{x}\otimes\boldsymbol{x})\mathbb{H}_{m-2}(T).
$$
Indeed, there are three types of shape function spaces in $\Sigma_{r,k,m}(T;\mathbb{M})$:
\begin{align*}	
\Sigma_{k-1,k,k}(T;\mathbb{M})&=\mathbb{R}^d\otimes(\mathbb{P}_{k-1}(T;\mathbb{R}^d)\oplus\boldsymbol{x}\mathbb{H}_{k-1}(T)), \\
\Sigma_{k,k,k}(T;\mathbb{M})&=\mathbb{P}_k(T;\mathbb{M}), \\
\Sigma_{k,k,k+1}(T;\mathbb{M})&=\mathbb{P}_k(T;\mathbb{M})\oplus(\boldsymbol{x}\otimes\boldsymbol{x})\mathbb{H}_{k-1}(T).
\end{align*}
Apparently $\boldsymbol{\tau}\boldsymbol{n}|_F\in\mathbb{P}_{r}(F;\mathbb{R}^d)$ for any $\boldsymbol{\tau}\in\Sigma_{r,k,m}(T;\mathbb{M})$ and $F\in\mathcal{F}(T)$, and $\div\Sigma_{r,k,m}(T;\mathbb{M}) = V_{k-1,m-1}(T;\mathbb R^d)$.

The DoFs for space $\Sigma_{r,k,m}(T;\mathbb{M})$ are given by
\begin{subequations}\label{div-M-dof}
\begin{align}
\label{div-M-dof1}
(\boldsymbol{\tau}\boldsymbol n , \boldsymbol{q})_F, & \quad\boldsymbol{q}\in \mathbb{P}_r(F;\mathbb{R}^d), \, F\in\mathcal{F}(T),\\
\label{div-M-dof2}
(\boldsymbol{\tau}, \boldsymbol{q})_T, & \quad \boldsymbol{q}\in \grad(V_{k-1,m-1}(T;\mathbb R^d)), \\
\label{div-M-dof3}
(\boldsymbol{\tau}, \boldsymbol{q})_T, & \quad  \boldsymbol{q}\in 
\mathbb{R}^d\otimes(\mathbb{P}_{r-1}(T;\mathbb{R}^d)\cap\ker(\cdot\boldsymbol{x})).
\end{align}
\end{subequations}

\begin{lemma}\label{lem:divMunisol}
The DoFs \eqref{div-M-dof} are unisolvent for space $\Sigma_{r,k,m}(T;\mathbb{M})$.
\end{lemma}
\begin{proof}
When $m=k$, space $\Sigma_{r,k,m}(T;\mathbb{M})$ and DoFs \eqref{div-M-dof} form the tensor-valued counterpart of the BDM element \cite{BrezziDouglasMarini1985,BrezziDouglasDuranFortin1987,nedelec1986new} and the RT element~\cite{RaviartThomas1977,Nedelec1980}. 
We only prove the unisolvence for the case $r=k$ and $m=k+1$.

By comparing DoFs \eqref{div-M-dof} with the DoFs of the BDM element, we see that the number of DoFs \eqref{div-M-dof} equals $\dim\mathbb{P}_k(T;\mathbb{M})+\dim\mathbb{H}_{k-1}(T)=\dim\Sigma_{r,k,m}(T;\mathbb{M})$. 

Assume $\boldsymbol{\tau}\in\Sigma_{r,k,m}(T;\mathbb{M})$ and all the DoFs \eqref{div-M-dof} vanish. The vanishing DoF~\eqref{div-M-dof1} implies $(\boldsymbol{\tau}\boldsymbol{n})|_{\partial T}=0$. This combined with the vanishing DoF \eqref{div-M-dof2} and integration by parts gives $\div\boldsymbol{\tau}=0$. Then by $\div\Sigma_{r,k,m}(T;\mathbb{M}) = \mathbb{P}_{k-1}(T;\mathbb{R}^d)\oplus\boldsymbol{x}\mathbb{H}_{k-1}(T)$, it holds $\boldsymbol{\tau}\in\mathbb{P}_{k}(T;\mathbb{M})$. Finally, we end the proof by applying the unisolvence of the BDM element.
\end{proof}

The global finite element space for tensors is defined as
\begin{align*}
\Sigma_{r,k,m,h}^{\div}:=\{\boldsymbol{\tau}\in L^{2}(\Omega;\mathbb{M})&:\boldsymbol{\tau}|_T\in \Sigma_{r,k,m}(T;\mathbb{M})~\textrm{for}~\textrm{each}~T\in\mathcal{T}_h; \\
&\qquad\quad\,\textrm{the~DoF \eqref{div-M-dof1} is single-valued}\}. 
\end{align*}
We have $\Sigma_{r,k,m,h}^{\div}\subset H(\div,\Omega;\mathbb{M})$. %, and $\div\Sigma_{r,k,m,h}^{\div} = \mathbb{P}_{k-1}(\mathcal{T}_h;\mathbb{R}^d)$.

\subsection{$H^1$-nonconforming virtual elements and some inequalities.}
For integers $k\geq1$ and $m=k,k+1$, the shape function space for the $H^1$-nonconforming virtual element (cf. \cite{ncvem2016,ChenHuangncvem2020}) is defined as
\begin{equation*}
V_{k,m}^{\rm VE}(T)=\{v\in H^1(T):\Delta v\in\mathbb{P}_{m-2}(T), \partial_n v|_F\in\mathbb{P}_{k-1}(F) \textrm{ for } F\in\mathcal{F}(T)\}.
\end{equation*}
Obviously $\mathbb{P}_{k}(T)\subseteq V_{k,m}^{\rm VE}(T)$, $V^{\rm VE}_{1,1}(T)=\mathbb{P}_{1}(T)$, and $V^{\rm VE}_{1,2}(T)=\mathbb{P}_{1}(T)\oplus\textrm{span}\{|\boldsymbol{x}|^2\}$, where $\mathrm{span}\{|\boldsymbol{x}|^2\}$ denotes the one-dimensional space spanned by $|\boldsymbol{x}|^2 = \boldsymbol{x} \cdot \boldsymbol{x}$.
The DoFs for $V_{k,m}^{\rm VE}(T)$ are given by
\begin{subequations}\label{ve-dof}
\begin{align}
\label{ve-dof1}
\dfrac{1}{|F|}(v,q)_F, & \quad q\in \mathbb{P}_{k-1}(F), \, F\in\mathcal{F}(T),\\
\label{ve-dof2}
\dfrac{1}{|T|}(v,q)_T, & \quad q\in \mathbb{P}_{m-2}(T).
\end{align}
\end{subequations}
For $k=1$ and $m=1$, it is exactly the Crouzeix--Raviart (CR) element \cite{CR1973}. 
For $k=1$ and $m=2$, it is the enriched Crouzeix-Raviart element \cite{HuMa2015,HuHuangLin2014}.

We first establish the norm equivalence \eqref{ve-infsup2} for the gradient of virtual element functions, which is crucial for stabilization-free virtual element methods \cite{ChenHuangWei2024vem}.

\begin{lemma}
For $T\in\mathcal{T}_h$,
it holds the norm equivalence
\begin{equation}
\label{ve-infsup2}
\|Q_{T}^{\rm div}\nabla v\|_{0,T}\eqsim\|\nabla v\|_{0,T} \quad \forall\, v\in V_{k,m}^{\rm VE}(T).
\end{equation}
\end{lemma}
\begin{proof}
The inequality $\|Q_{T}^{\rm div}\nabla v\|_{0,T}\leq\|\nabla v\|_{0,T}$ follows directly from the fact that $Q_{T}^{\rm div}$ is the $L^2$-orthogonal projection operator. Hence, it remains to establish the reverse inequality, i.e., the inf-sup condition
\begin{equation}
\label{ve-infsup1}
\|\nabla v\|_{0,T}\lesssim \sup_{\boldsymbol{w}\in V_{k-1,m-1}(T;\mathbb R^d)}\dfrac{(\boldsymbol{w},\nabla v)_T}{\|\boldsymbol{w}\|_{0,T}} \quad \forall\, v\in V_{k,m}^{\rm VE}(T).
\end{equation}
The norm equivalence \eqref{ve-infsup2} holds for $k=1$, as $Q_{T}^{\rm div}\nabla v=\nabla v$ in this case.

Now we consider the case $k\geq2$. We follow the proof of Lemma 4.4 in~\cite{ChenHuangWei2024vem}. 
Recall that $Q_{m-2,T}$ and $Q_{k-1,F}$ are standard $L^2$-projection operators onto $\mathbb{P}_{m-2}(T)$  and $\mathbb{P}_{k-1}(F)$, respectively.
Without loss of generality, assume $v\in V_{k,m}^{\rm VE}(T)\cap L_0^2(T)$. Then $Q_{m-2,T}v\in \mathbb P_{m-2}(T)\cap L_0^2(T)$.
Based on DoFs \eqref{div-dof}, take $\boldsymbol{w}\in V_{k-1,m-1}(T;\mathbb R^d)$ such that
\begin{align*}
(\boldsymbol{w}\cdot\boldsymbol{n}, q)_F&=h_T^{-1}(v, q)_F  \qquad\qquad\quad\quad\;\;\,\forall\,q\in\mathbb P_{k-1}(F),  F\in\mathcal F(T), \\
(\boldsymbol{w}, \nabla q)_T&=h_T^{-2}(v, q)_T+h_T^{-1}(v, q)_F  \quad\forall\,q\in\mathbb P_{m-2}(T)/\mathbb R, \\
(\boldsymbol{w}, \boldsymbol{q})_T&=0  \qquad\qquad\qquad\qquad\qquad\;\;\;\forall\,\boldsymbol{q}\in  \mathbb{P}_{k-2}(T;\mathbb{R}^d)\cap\ker(\cdot\boldsymbol{x}). 
\end{align*}
Then $(\boldsymbol{w}\cdot\boldsymbol{n})|_{F}=h_T^{-1}Q_{k-1,F}v$ for $F\in\mathcal F(T)$. 
By a scaling argument, we have
\begin{equation}\label{eq:20250104}
\|\boldsymbol{w}\|_{0,T}^2\lesssim h_T^{-2}\|Q_{m-2,T}v\|_{0,T}^2+\sum_{F\in\mathcal F(T)}h_T^{-1}\|Q_{k-1,F}v\|_{0,F}^2.
\end{equation}
Apply integration by parts to get
\begin{align*}
(\boldsymbol{w}, \nabla v)_T&=-(\div\boldsymbol{w}, v)_T + (\boldsymbol{w}\cdot\boldsymbol{n}, v)_{\partial T} \\
&=-(\div\boldsymbol{w}, Q_{m-2,T}v)_T + \sum_{F\in\mathcal F(T)}h_T^{-1}\|Q_{k-1,F}v\|_{0,F}^2 \\
&=h_T^{-2}\|Q_{m-2,T}v\|_{0,T}^2 + \sum_{F\in\mathcal F(T)}h_T^{-1}\|Q_{k-1,F}v\|_{0,F}^2.
\end{align*}
On the other hand, adopting a scaling argument yields
the following inverse inequality
\begin{equation}
\label{ve-inv}
|v|_{1,T}\lesssim h^{-1}_T\|v\|_{0,T} \quad \forall\, v\in V_{k,m}^{\rm VE}(T),
\end{equation}
and the $L^2$ norm equivalence
\begin{equation}
\label{ve-normeq}
\|v\|^2_{0,T}\eqsim \|Q_{m-2,T}v\|^2_{0,T}+\sum_{F\in\mathcal{F}(T)}h_F\|Q_{k-1,F}v\|^2_{0,F} \quad \forall\, v\in V_{k,m}^{\rm VE}(T).
\end{equation}
By \eqref{ve-inv}-\eqref{ve-normeq}, we get
\begin{equation*}
\|\nabla v\|_{0,T}^{2}\lesssim h_T^{-2}\|Q_{m-2,T}v\|_{0,T}^2+\sum_{F\in\mathcal F(T)}h_T^{-1}\|Q_{k-1,F}v\|_{0,F}^2= (\boldsymbol{w}, \nabla v)_T.
\end{equation*}
This together with \eqref{eq:20250104} implies the inf-sup condition \eqref{ve-infsup1}.
\end{proof}

We next present two estimates to be used in the error analysis in the following section.
\begin{lemma}
Let integers $k\geq1$ and $m=k,k+1$.
For $T\in\mathcal{T}_h$,
we have
\begin{equation}
\label{H1ineq}
\sum_{F\in\mathcal{F}(T)}h_F^{-1/2}\|Q_{m-2,T}v-Q_{k-1,F}v\|_{0,F}\lesssim |v|_{1,T} \quad \forall\, v\in V_{k,m}^{\rm VE}(T)
\end{equation}
for $m \geq 2$, and
\begin{equation}
\label{H2ineq}
\sum_{F\in\mathcal{F}(T)}h_F^{-1/2}\|Q_{m-2,T}v-Q_{k-1,F}v\|_{0,F}\lesssim h_T|Q_{T}^{\rm div}\nabla v|_{1,T} \quad \forall\, v\in V_{k,m}^{\rm VE}(T)
\end{equation}
for $m \geq 3$.
\end{lemma}
\begin{proof}
The inequality \eqref{H1ineq} is a direct consequence of 
\begin{equation*}
\|Q_{m-2,T}v-Q_{k-1,F}v\|_{0,F}=\|Q_{k-1,F}(Q_{m-2,T}v-v)\|_{0,F}\leq \|Q_{m-2,T}v-v\|_{0,F},
\end{equation*}
the trace inequality and the error estimate of the $L^2$-projection $Q_{m-2,T}v$.

Next we prove the inequality \eqref{H2ineq} for $m\geq3$. For $v\in V_{k,m}^{\rm VE}(T)$, take $\boldsymbol{w}\in V_{k-1,m-1}(T;\mathbb R^d)$ such that all the DoFs \eqref{div-dof} vanish except
\begin{align*}
(\boldsymbol{w}\cdot\boldsymbol{n}, q)_F&=h^{-1}_F(Q_{m-2,T}v-Q_{k-1,F}v,q)_F  \quad\;\;\,\forall\,q\in\mathbb P_{k-1}(F),  F\in\mathcal F(T). 
\end{align*}
Then $(\boldsymbol{w}\cdot\boldsymbol{n})|_{F}=h^{-1}_F(Q_{m-2,T}v-Q_{k-1,F}v)|_F$ for $F\in\mathcal F(T)$. 
By a scaling argument, we have
\begin{equation}
\label{H2ineq-pf1}
\|\boldsymbol{w}\|_{0,T}\lesssim \sum_{F\in\mathcal{F}(T)}h_F^{-1/2}\|Q_{m-2,T}v-Q_{k-1,F}v\|_{0,F}.
\end{equation}
On the other side, by the vanishing DoF \eqref{div-dof2} of $\boldsymbol{w}$ and \eqref{eq:divDoFintfullk}, it follows that
\begin{equation*}
(\boldsymbol{w},Q_{m-3,T}\nabla v)_T=(\boldsymbol{w},\nabla Q_{m-2,T}v)_T=0.
\end{equation*}
Apply integration by parts to acquire
\begin{align*}
(\boldsymbol{w}, Q_{m-3,T}\nabla v-Q_{T}^{\rm div}\nabla v)_T&=(\boldsymbol{w}, \nabla(Q_{m-2,T}v-v))_T
=(\boldsymbol{w}\cdot\boldsymbol{n}, Q_{m-2,T}v-v)_{\partial T} \\
& =\sum_{F\in\mathcal{F}(T)}h^{-1}_F\|Q_{m-2,T}v-Q_{k-1,F}v\|^2_{0,F}.
\end{align*}
Then
\begin{equation*}
\sum_{F\in\mathcal{F}(T)}h^{-1}_F\|Q_{m-2,T}v-Q_{k-1,F}v\|^2_{0,F}\leq\|Q_{m-3,T}\nabla v-Q_{T}^{\rm div}\nabla v\|_{0,T}\|\boldsymbol{w}\|_{0,T}.
\end{equation*}
This combined with \eqref{H2ineq-pf1} yields
\begin{equation*}
\sum_{F\in\mathcal{F}(T)}h_F^{-1/2}\|Q_{m-2,T}v-Q_{k-1,F}v\|_{0,F}\lesssim \|Q_{T}^{\rm div}\nabla v-Q_{m-3,T}(Q_{T}^{\rm div}\nabla v)\|_{0,T}.
\end{equation*}
Therefore, \eqref{H2ineq} holds from the error estimate of the $L^2$-projector $Q_{m-3,T}$. 
\end{proof}

Due to the regularity of the Neumann problem employed within the virtual element space $ V_{k,m}^{\rm VE}(T) $, the space $ V_{k,m}^{\rm VE}(T) $ is not necessarily a subspace of $ H^2(T) $. Consequently, we use $ |Q_{T}^{\rm div} \nabla v|_{1,T} $ instead of $ |v|_{2,T} $ in the right-hand side of \eqref{H2ineq}.

Next we define the global virtual element space $\mathring{V}_{k,m,h}^{\rm VE}$ by
\begin{equation}\label{eq:VEMspace}
\begin{aligned}
\mathring{V}_{k,m,h}^{\rm VE} &= \{v\in L^2(\Omega):v|_T\in V_{k,m}^{\rm VE}(T) \textrm{ for each }T\in\mathcal{T}_h; \textrm{ DoF \eqref{ve-dof1} is } \\
&\qquad\quad\;\;\textrm{ single-valued across each face in $\mathring{\mathcal{F}}_h$, and vanish on $\partial\Omega$}\}.
\end{aligned}
\end{equation}
The virtual element space $\mathring{V}_{k,m,h}^{\rm VE}$ has the weak continuity
\begin{equation}\label{weak-c}
([\![v]\!],q)_F = 0 \quad \forall\, v\in\mathring{V}_{k,m,h}^{\rm VE}, q\in\mathbb{P}_{k-1}(F), F\in\mathcal{F}_h.
\end{equation}
Equip the space $\mathring{V}_{k,m,h}^{\rm VE}$ with the discrete $H^2$ seminorm
\begin{equation*}
{\interleave v\interleave}_{2,h}^2:=|Q_h^{\rm div}\nabla_h v|^2_{1,h}+\sum_{F\in\mathcal{F}_h}h_F^{-1}\|[\![Q_h^{\rm div}\nabla_hv]\!]\|^2_{0,F}.
\end{equation*}
By the broken Poincar\'e inequality \cite[(1.8)]{MR1974504} and the norm equivalence \eqref{ve-infsup2}, 
${\interleave \cdot\interleave}_{2,h}$ is a norm on the virtual
element space $\mathring{V}_{k,m,h}^{\rm VE}$.
When $k=1$ and $m=1,2$, we have
\begin{equation*}
{\interleave v\interleave}_{2,h}^2=|v|^2_{2,h}+\sum_{F\in\mathcal{F}_h}h_F^{-1}\|[\![\nabla_hv]\!]\|^2_{0,F} \quad\forall\,v\in\mathring{V}_{1,m,h}^{\rm VE}.
\end{equation*}

\begin{lemma}
For $k=1$ and $m=1,2$, it holds
\begin{equation}
\label{cr-jump}
\sum_{F\in\mathcal{F}_h}h^{-1}_F\|[\![v_h]\!]\|_{0,F}^2\lesssim \min\{|v_h|^{2}_{1,h}, h^{2}{\interleave v_h\interleave}_{2,h}^{2}\}  \quad\;\; \forall\, v_h\in \mathring{V}_{1,m,h}^{\rm VE}.
\end{equation}
\end{lemma}
\begin{proof}
By the weak continuity \eqref{weak-c}, we have for any $v_h\in\mathring{V}_{1,m,h}^{\rm VE}$ that (cf. \cite[Remark 1.1]{MR1974504} and \cite[Lemma 3.3]{Wang2001})
\begin{align*}
\sum_{F\in\mathcal{F}_h}h^{-1}_F\|[\![v_h]\!]\|_{0,F}^2 &\lesssim |v_h|_{1,h}^2, \\
\sum_{F\in\mathcal{F}_h}h^{-1}_F\|[\![v_h]\!]\|_{0,F}^2&=\sum_{F\in\mathcal{F}_h}h^{-1}_F\|[\![v_h]\!]-Q_{0,F}[\![v_h]\!]\|_{0,F}^2\lesssim \sum_{F\in\mathcal{F}_h}h_F\|[\![\nabla_Fv_h]\!]\|_{0,F}^2\\
&\leq\sum_{F\in\mathcal{F}_h}h_F\|[\![\nabla_h v_h]\!]\|_{0,F}^2\lesssim h^2{\interleave v_h\interleave}_{2,h}^2. 
\end{align*}
Then a combination of the last two inequalities yields \eqref{cr-jump}. 
\end{proof}

\subsection{Interpolation operators}
Let $I^{\rm VE}_{h}: H_0^1(\Omega)\rightarrow\mathring{V}_{k,m,h}^{\rm VE}$ be the global canonical interpolation operator based on the DoFs \eqref{ve-dof}.
By a scaling argument, we have for any $0\leq s\leq k$ and $T\in\mathcal{T}_h$ that (cf. \cite[Lemma 5.3]{ChenHuangncvem2020})
\begin{equation}
\label{ve-int-err}
\|v-I^{\rm VE}_hv\|_{0,T}+h_T|v-I^{\rm VE}_hv|_{1,T}\lesssim h_T^{s+1}|v|_{s+1,T}\quad \forall\,v\in H_0^1(\Omega)\cap H^{s+1}(\mathcal{T}_h).
\end{equation}

In the following, we show the orthogonality property of the interpolation operator $I^{\rm VE}_h$.
\begin{lemma}
Let integers $k\geq1$ and $m=k,k+1$. It holds
\begin{equation}
\label{exchange}
Q_h^{\rm div}\nabla_h(I^{\rm VE}_hv) = Q_h^{\rm div}\nabla v\qquad\forall\,v\in H_0^1(\Omega).
\end{equation}
\end{lemma}
\begin{proof}
For $T\in\mathcal{T}_h$ and $\boldsymbol{q}\in V_{k-1,m-1}(T;\mathbb R^d)$,
apply integration by parts to get
\begin{align*}
(\nabla(v-I^{\rm VE}_hv),\boldsymbol{q})_T=(v-I^{\rm VE}_hv, \boldsymbol{q}\cdot\boldsymbol{n})_{\partial T}-(v-I^{\rm VE}_hv, \div\boldsymbol{q})_T=0,
\end{align*} 
which implies \eqref{exchange}.
\end{proof}

Let $I^{\div}_{h}:H^1(\Omega;\mathbb{R}^{d})\rightarrow V_{k-1,m-1,h}^{\div}$ be the canonical interpolation operator based on DoFs~\eqref{div-dof}, where 
\begin{equation*}  
V_{k-1,m-1,h}^{\div} := \{ \boldsymbol{v} \in H(\div, \Omega) : \boldsymbol{v}|_T \in V_{k-1,m-1}(T;\mathbb R^d) \ \text{for each} \ T \in \mathcal{T}_h \}.  
\end{equation*}
Next, we present some properties of the interpolation operator $I^{\div}_{h}$.
\begin{lemma}
Let integers $k \geq 1$ and $m \geq 2$, where $m = k$ or $m = k+1$.
For any $\boldsymbol{w}\in H^1(\Omega;\mathbb{R}^{d})$, we have
\begin{align}
\label{BDMint1}
\div I^{\div}_{h}\boldsymbol{w} &= Q_{m-2,h}\div\boldsymbol{w},\\
\label{BDMint2}
(\div\boldsymbol{w},Q_{m-2,h}v_h) &= -(I^{\div}_{h}\boldsymbol{w}, \nabla_hv_h) \qquad\forall\, v_h\in \mathring{V}_{k,m,h}^{\rm VE},
\end{align}
and
\begin{align}
\begin{split}
\label{BDMint3}
&(Q_h^{\rm div}\boldsymbol{w}-I^{\div}_{h}\boldsymbol{w},\nabla_hv_h) \\
=& \sum_{T\in \mathcal{T}_h}\sum_{F\in\mathcal{F}(T)}((Q_h^{\rm div}\boldsymbol{w}-\boldsymbol{w})\cdot\boldsymbol{n},Q_{k-1,F}v_h-Q_{m-2,T}v_h)_{F} \quad\forall\, v_h\in \mathring{V}_{k,m,h}^{\rm VE}.
\end{split}
\end{align}
\end{lemma}
\begin{proof}
The property \eqref{BDMint1} can be found in \cite[Theorem 5.2]{ArnoldFalkWinther2006} and \cite[Page 89]{Arnold2018}.
Employing \eqref{BDMint1}, integration by parts and $I^{\div}_{h} \boldsymbol{w} \in H(\div, \Omega)$, we obtain for $v_h\in\mathring{V}_{k,m,h}^{\rm VE}$ that
\begin{align*}
(\div\boldsymbol{w},Q_{m-2,h}v_h) &= (Q_{m-2,h}\div\boldsymbol{w},v_h) = (\div I^{\div}_{h}\boldsymbol{w},v_h)\\
& =-(I^{\div}_{h}\boldsymbol{w},\nabla_hv_h) + \sum_{F\in\mathcal{F}_h}(I^{\div}_{h}\boldsymbol{w}\cdot\boldsymbol{n}, [\![v_h]\!])_F.
\end{align*}
Hence, \eqref{BDMint2} holds from the weak continuity \eqref{weak-c}.

Set $\boldsymbol{\phi}=Q_h^{\rm div}\boldsymbol{w}-I^{\div}_{h}\boldsymbol{w}$ for simplicity. Applying integration by parts twice, we get for $T\in\mathcal{T}_h$ that
\begin{align*}
(\boldsymbol{\phi},\nabla v_h)_T
&= -(\div\boldsymbol{\phi}, Q_{m-2,T}v_h)_T+\sum_{F\in\mathcal{F}(T)}(\boldsymbol{\phi}\cdot\boldsymbol{n},Q_{k-1,F}v_h)_F\\
&=(\boldsymbol{\phi},\nabla Q_{m-2,T}v_h)_T+\sum_{F\in\mathcal{F}(T)}(\boldsymbol{\phi}\cdot\boldsymbol{n},Q_{k-1,F}v_h-Q_{m-2,T}v_h)_F.
\end{align*}
Then \eqref{BDMint3} holds from the definitions of $I^{\div}_{h}$ and $Q_h^{\rm div}$.
\end{proof}

An $L^2$-bounded commuting projection operator onto the tensor-valued finite element space $\Sigma_{r,k,m,h}^{\div}$ is essential for the robust analysis in the next section. To this end,
let $\widetilde{\Pi}_{h}: L^2(\Omega;\mathbb{M})\rightarrow\Sigma_{k+1,k+1,k+1,h}^{\div}$ denote the tensor-valued counterpart of the $L^2$-bounded commuting projection operator devised in  \cite[(5.2)]{ArnoldGuzman2021}. By \cite[Theorem 3.1]{ArnoldGuzman2021} and a scaling argument,
we have
\begin{align}
\label{tildePihProp1}
&\qquad\quad\;\, \div(\widetilde{\Pi}_{h}\boldsymbol{\tau}) = Q_{k,h}\div\boldsymbol{\tau}\quad\;\forall\,\boldsymbol{\tau}\in H^1(T;\mathbb{M}), T\in\mathcal{T}_h,\\
\label{tildePihProp2}
&\sum_{T\in\mathcal{T}_h}h_T^{-2s}\|\boldsymbol{\tau}-\widetilde{\Pi}_{h}\boldsymbol{\tau}\|_{0,T}^2\lesssim |\boldsymbol{\tau}|_{s}^2\quad\,\forall\,\boldsymbol{\tau}\in H^{s}(\Omega;\mathbb{M}), 0\leq s\leq k+2.
\end{align}
Define the operator $\widehat{\Pi}_{h}: H^1(\mathcal{T}_h;\mathbb{M})\to\mathbb P_{k+1}(\mathcal{T}_h;\mathbb{M})$ as follows: for any $\boldsymbol{\tau}\in H^1(\mathcal{T}_h;\mathbb{M})$ and $T\in\mathcal{T}_h$, $(\widehat{\Pi}_{h}\boldsymbol{\tau})|_T\in \Sigma_{r,k,m}(T;\mathbb{M})$ is the canonical interpolation of $\boldsymbol{\tau}|_T$ based on DoFs \eqref{div-M-dof}. By a scaling argument, we have
\begin{equation}\label{widehatPihProp1}
\|\boldsymbol{\tau}-\widehat{\Pi}_{h}\boldsymbol{\tau}\|_{0,T}\lesssim h_T|\boldsymbol{\tau}|_{1,T}\quad\,\forall\,\boldsymbol{\tau}\in H^1(\mathcal{T}_h;\mathbb{M}), T\in\mathcal{T}_h.
\end{equation}

Using the $L^2$-bounded commuting projection operator $\widetilde{\Pi}_{h}$ and the interpolation operator $\widehat{\Pi}_{h}$, we define an $L^2$-bounded projection operator $\Pi_{h}: L^2(\Omega; \mathbb{M}) \to \Sigma_{r,k,m,h}^{\div}$ as  
$$  
\Pi_{h} \boldsymbol{\tau} := \widehat{\Pi}_{h}(\widetilde{\Pi}_{h} \boldsymbol{\tau})\quad\forall\,\boldsymbol{\tau} \in L^2(\Omega; \mathbb{M}).  
$$  
Since $\widetilde{\Pi}_{h} \boldsymbol{\tau} \in H(\div, \Omega;\mathbb M)$, it follows that $\Pi_{h} \boldsymbol{\tau} \in H(\div, \Omega;\mathbb M)$ as well.

At the end of this section, we present some properties of the operator $\Pi_{h}$.
\begin{lemma}
We have
\begin{align}
\label{div-M-int1}
&\qquad\quad\; \div(\Pi_{h}\boldsymbol{\tau}) = Q_h^{\rm div}\div\boldsymbol{\tau}\quad\;\forall\,\boldsymbol{\tau}\in H^1(T;\mathbb{M}),\\
\label{div-M-int2}
&\sum_{T\in\mathcal{T}_h}h_T^{-2s}\|\boldsymbol{\tau}-\Pi_{h}\boldsymbol{\tau}\|_{0,T}^2\lesssim |\boldsymbol{\tau}|_{s}^2\quad\,\forall\,\boldsymbol{\tau}\in H^{s}(\Omega;\mathbb{M}), 0\leq s\leq r+1.
% \|\boldsymbol{\tau}-\Pi_{h}\boldsymbol{\tau}\|_{0,T}&\lesssim h^{r+1}|\boldsymbol{\tau}|_{r+1,T}\quad\,\forall\,\boldsymbol{\tau}\in H^{r+1}(T;\mathbb{M}). %,\\
% \label{div-M-int3}
% \|\div\boldsymbol{\tau}-\div\Pi_{h}\boldsymbol{\tau}\|_{0,T}&\lesssim h^{k}|\div\boldsymbol{\tau}|_{k,T}\quad\;\,\forall\,\boldsymbol{\tau}\in H^{k+1}(T;\mathbb{M}).
\end{align}
\end{lemma}
\begin{proof}
Due to the definition of $\Pi_{h}\boldsymbol{\tau}$ and integration by parts,
\begin{equation*}
(\div(\Pi_{h}\boldsymbol{\tau}-\widetilde{\Pi}_{h}\boldsymbol{\tau}), \boldsymbol{q})_T=0\quad\forall\,\boldsymbol{q}\in V_{k-1,m-1}(T;\mathbb R^d), T\in\mathcal{T}_h.
\end{equation*}
This means $\div(\Pi_{h}\boldsymbol{\tau})=Q_h^{\rm div}\div(\widetilde{\Pi}_{h}\boldsymbol{\tau})$. Then \eqref{div-M-int1} holds from \eqref{tildePihProp1}.

Apply the estimate \eqref{widehatPihProp1} of operator $\widehat{\Pi}_{h}$ and the inverse inequality to get
\begin{align*}
\|\Pi_{h}\boldsymbol{\tau}-\widetilde{\Pi}_{h}\boldsymbol{\tau}\|_{0,T}&=\|\widetilde{\Pi}_{h}\boldsymbol{\tau}-Q_{r,h}\boldsymbol{\tau}-\widehat{\Pi}_{h}(\widetilde{\Pi}_{h}\boldsymbol{\tau}-Q_{r,h}\boldsymbol{\tau})\|_{0,T} \\
&\lesssim h_T|\widetilde{\Pi}_{h}\boldsymbol{\tau}-Q_{r,h}\boldsymbol{\tau}|_{1,T}\lesssim \|\widetilde{\Pi}_{h}\boldsymbol{\tau}-Q_{r,h}\boldsymbol{\tau}\|_{0,T}\\
&\lesssim \|\boldsymbol{\tau}-\widetilde{\Pi}_{h}\boldsymbol{\tau}\|_{0,T}+ \|\boldsymbol{\tau}-Q_{r,h}\boldsymbol{\tau}\|_{0,T}.
\end{align*}
Therefore, the estimate \eqref{div-M-int2} follows from \eqref{tildePihProp2} and the estimate of $Q_{r,h}$.
\end{proof}

%%%%%%%%%%%%%%%%%%%%%%%%%%%%%%%%%%%%%%%%
%% VEM
%%%%%%%%%%%%%%%%%%%%%%%%%%%%%%%%%%%%%%%%
\section{Robust Mixed Methods}\label{sec4}

In this section, we develop and analyze a family of mixed methods for the fourth-order elliptic singular perturbation problem~\eqref{FSP1}.

\subsection{Mixed methods}
% For integers $k\geq1$, $r\geq0$ and $m\geq2$ satisfying \eqref{eq:conditionkrm}.
Based on the mixed formulation \eqref{FSP-mix}, we propose the following mixed method for the fourth-order elliptic singular perturbation problem~\eqref{FSP1}:
find $\boldsymbol{\sigma}_h\in 	\Sigma_{r,k,m,h}^{\div}$ and $u_h\in\mathring{V}_{k,m,h}^{\rm VE}$ such that
\begin{subequations}\label{FSP-MVEM}
\begin{align}
\label{FSP-MVEM1}
a(\boldsymbol{\sigma}_h,\boldsymbol{\tau}_h)+b_h(\boldsymbol{\tau}_h,u_h) &= 0 \qquad\qquad\;\, \forall\, \boldsymbol{\tau}_h\in  \Sigma_{r,k,m,h}^{\div}, \\
\label{FSP-MVEM2}
b_h(\boldsymbol{\sigma}_h,v_h)-c_h(u_h,v_h)&= -\langle\!\langle f, v_h\rangle\!\rangle \quad \forall\, v_h\in \mathring{V}_{k,m,h}^{\rm VE},
\end{align}
\end{subequations} 
where
\begin{align}\label{bhch}
& b_h(\boldsymbol{\tau}_h,v_h):=(\div\boldsymbol{\tau}_h, Q_h^{\rm div}\nabla_h v_h),\quad c_h(u_h,v_h):=(Q_h^{\rm div}\nabla_h u_h,Q_h^{\rm div}\nabla_h v_h), \\	
\notag
&\qquad\qquad\qquad\qquad \langle\!\langle f, v_h\rangle\!\rangle = 
\begin{cases}
(f, v_h), \quad &\textrm{for}~k=1, \\
(f, Q_{m-2,h}v_h), \quad &\textrm{for}~k\geq2.
\end{cases}
\end{align}
By the fact that $\div\boldsymbol{\tau}_h$ is in the range of the operator $Q_h^{\rm div}$ for $\boldsymbol{\tau}_h\in \Sigma_{r,k,m,h}^{\div}$, it holds that
\begin{equation*}
b_h(\boldsymbol{\tau}_h,v_h)=(\div\boldsymbol{\tau}_h, \nabla_h v_h)\quad\forall\,\boldsymbol{\tau}_h\in \Sigma_{r,k,m,h}^{\div}, v_h\in H^1(\mathcal{T}_h).	
\end{equation*}
The discrete method \eqref{FSP-MVEM} is a mixed finite element method for $k=1$, and the projector $Q_h^{\rm div}$ in $b_h(\boldsymbol{\tau}_h,v_h)$ and $c_h(u_h,v_h)$ can be omitted. For $k\geq2$,
the discrete method~\eqref{FSP-MVEM} is a stablization-free mixed finite-virtual element method.

Equip the space $\mathring{V}_{k,m,h}^{\rm VE}$ with the discrete parameter-dependent norm
$${\interleave v_h\interleave}^2_{\varepsilon,h}:=\varepsilon^2{\interleave v_h\interleave}^2_{2,h}+|v_h|^2_{1,h}.$$

\begin{lemma}
There holds the discrete inf-sup condition
\begin{equation}
\label{inf-supH2}
{\interleave v_h\interleave}_{2,h}\lesssim\sup_{\boldsymbol{\tau}_h\in \Sigma_{r,k,m,h}^{\div}}\dfrac{(\div\boldsymbol{\tau}_h, \nabla_h v_h)}{\|\boldsymbol{\tau}_h\|_0} \quad \forall\, v_h\in\mathring{V}_{k,m,h}^{\rm VE}.
\end{equation}
\end{lemma}
\begin{proof}
Let $\boldsymbol{\tau}_h\in\Sigma_{r,k,m,h}^{\div}$ such that all the DoFs \eqref{div-M-dof} vanish except 
\begin{align*}
(\boldsymbol{\tau}_h\boldsymbol{n}, \boldsymbol{q})_F&=h_F^{-1}([\![Q_h^{\rm div}\nabla_h v_h]\!], \boldsymbol{q})_F \quad\quad \forall\, \boldsymbol{q}\in\mathbb{P}_{r}(F;\mathbb{R}^d), \\
(\boldsymbol{\tau}_h,\boldsymbol{q})_T &= -(\grad(Q_h^{\rm div}\nabla_hv_h),\boldsymbol{q})_T \quad \forall\, \boldsymbol{q}\in\grad(V_{k-1,m-1}(T;\mathbb R^d))
\end{align*}
for each $F\in\mathcal{F}_h$ and $T\in\mathcal{T}_h$. By a scaling argument, we have
\begin{align}
\label{inf-supH2-pf2}
\|\boldsymbol{\tau}_h\|^2_0\lesssim|Q_h^{\rm div}\nabla_h v_h|^2_{1,h}+\sum_{F\in\mathcal{F}_h}h_F^{-1}\|[\![Q_h^{\rm div}\nabla_hv_h]\!]\|^2_{0,F}={\interleave v_h\interleave}_{2,h}^2.
\end{align}
Applying integration by parts, we get
\begin{align*}
(\div\boldsymbol{\tau}_h, Q_h^{\rm div}\nabla_h v_h)&=-\sum_{T\in\mathcal{T}_h}(\boldsymbol{\tau}_h, \grad(Q_h^{\rm div}\nabla_h v_h))_T+\sum_{F\in\mathcal{F}_h}(\boldsymbol{\tau}_h\boldsymbol{n},[\![Q_h^{\rm div}\nabla_h v_h]\!])_F\\
& = |Q_h^{\rm div}\nabla_h v_h|^2_{1,h}+\sum_{F\in\mathcal{F}_h}h_F^{-1}\|[\![Q_h^{\rm div}\nabla_hv_h]\!]\|^2_{0,F}={\interleave v_h\interleave}_{2,h}^2,
\end{align*}
which together with \eqref{inf-supH2-pf2} implies \eqref{inf-supH2}.
\end{proof}

\begin{theorem}\label{thm:FSP-MVEM}
The mixed method \eqref{FSP-MVEM} is well-posed.
It holds the discrete stability
\begin{align}
\label{discretestability}
&\varepsilon^{-1}\|\boldsymbol{\sigma}_h\|_0+\interleave u_h\interleave_{\varepsilon,h} \\
\notag
&\quad \lesssim\sup_{\boldsymbol{\tau}_h\in \Sigma_{r,k,m,h}^{\div}, v_h\in \mathring{V}_{k,m,h}^{\rm VE}}
\frac{a(\boldsymbol{\sigma}_h, \boldsymbol{\tau}_h)+b_h(\boldsymbol{\tau}_h, u_h)+b_h(\boldsymbol{\sigma}_h, v_h)-c_h(u_h, v_h)}{\varepsilon^{-1}\|\boldsymbol{\tau}_h\|_0+\interleave v_h\interleave_{\varepsilon,h}}
\end{align}
for any $\boldsymbol{\sigma}_h\in 	\Sigma_{r,k,m,h}^{\div}$ and $u_h\in \mathring{V}_{k,m,h}^{\rm VE}$.
When $k\geq2$, $\div\boldsymbol{\sigma}_h-Q_h^{\rm div}\nabla_h u_h\in H(\div,\Omega)$.
\end{theorem}
\begin{proof}
It follows from integration by parts and the inverse inequality that
\begin{equation*}
b_h(\boldsymbol{\tau}_h, v_h)\lesssim \|\boldsymbol{\tau}_h\|_0{\interleave v_h\interleave}_{2,h}\qquad\forall\,\boldsymbol{\tau}_h\in \Sigma_{r,k,m,h}^{\div}, v_h\in \mathring{V}_{k,m,h}^{\rm VE}.
\end{equation*}
Then by \eqref{ve-infsup2} and the discrete inf-sup condition \eqref{inf-supH2}, we have
\begin{align*}
\varepsilon^{-2}\|\boldsymbol{\tau}_h\|^2_0\leq a(\boldsymbol{\tau}_h, \boldsymbol{\tau}_h)+ \sup_{v\in \mathring{V}_{k,m,h}^{\rm VE}}\frac{b_h^2(\boldsymbol{\tau}_h, v_h)}{{\interleave v_h\interleave}_{\varepsilon,h}^2} \lesssim\varepsilon^{-2}\|\boldsymbol{\tau}_h\|^2_0 \qquad &\forall\, \boldsymbol{\tau}_h\in 	\Sigma_{r,k,m,h}^{\div},\\
{\interleave v_h\interleave}^2_{\varepsilon,h}\lesssim c_h(v_h, v_h) + \sup_{\boldsymbol{\tau}_h\in \Sigma_{r,k,m,h}^{\div}}\frac{b_h^2(\boldsymbol{\tau}_h, v_h)}{\varepsilon^{-2}\|\boldsymbol{\tau}_h\|_0^2} \lesssim{\interleave v_h\interleave}^2_{\varepsilon,h} \qquad &\forall\, v_h\in \mathring{V}_{k,m,h}^{\rm VE}.
\end{align*}
Apply the Zulehner Theory \cite[Theorem 2.6]{zulehner2011} to conclude the discrete stability \eqref{discretestability} and the unisolvence of the mixed method \eqref{FSP-MVEM}.

When $k\geq2$, we can choose $v_h\in \mathring{V}_{k,m,h}^{\rm VE}$ such that $Q_{m-2,h}v_h=0$, i.e. DoF \eqref{ve-dof2} vanishes.
Substituting this $v_h$ into \eqref{FSP-MVEM2}, we derive $\div\boldsymbol{\sigma}_h-Q_h^{\rm div}\nabla_h u_h\in H(\div,\Omega)$ from integration by parts and DoF \eqref{ve-dof1}.
\end{proof}

\begin{remark}\rm
For $k \geq d+1$, symmetric $H(\div)$-conforming tensor-valued finite elements in~\cite{ChenHuang2022div,ChenHuang2024,ArnoldWinther2002,AdamsCockburn2005,ArnoldAwanouWinther2008,HuZhang2016,Hu2015a,HuZhang2015} can be employed to discretize $\boldsymbol{\sigma}$. The constraint $k\geq d+1$ can be partially relaxed by using low-order symmetric tensor-valued finite elements proposed in~\cite{HuangZhangZhouZhu2024}.
\end{remark}

\begin{remark}\rm
Although the solution $\boldsymbol{\sigma}_h$ of the mixed method \eqref{FSP-MVEM} is not necessarily symmetric, its symmetric part $\sym \boldsymbol{\sigma}_h$ still provides an effective approximation of $\boldsymbol{\sigma}$. Notably, $\sym \boldsymbol{\sigma}_h$ is continuous in the normal-normal component, similar to those in~\cite{ChenHuang2025,Hellan1967,Herrmann1967,Johnson1973}.
\end{remark}

\subsection{Error analysis}
We begin with the following estimates of the consistency errors.
\begin{lemma}
Let $(\boldsymbol{\sigma},u)\in H^{-1}(\div\div,\Omega; \mathbb{M})\times H_0^1(\Omega)$ be the solution of problem~\eqref{FSP-mix}. Assume $\boldsymbol{\sigma}\in H^{k+1}(\Omega;\mathbb{M})$ and $u\in H^{k+1}(\Omega)$. We have the consistency error
\begin{equation}
\label{consistencyerr-1}
b_h(\boldsymbol{\sigma},v_h)-c_h(u,v_h)+\langle\!\langle f, v_h\rangle\!\rangle\lesssim h^k(|\boldsymbol{\sigma}|_{k+1}+|u|_{k+1}){|v_h|}_{1,h}\quad \forall\,v_h\in \mathring{V}_{k,m,h}^{\rm VE}.
\end{equation}
If $(k, m)\neq(2,2)$, we have
\begin{equation}
\label{consistencyerr-2}
b_h(\boldsymbol{\sigma},v_h)-c_h(u,v_h)+\langle\!\langle f, v_h\rangle\!\rangle\lesssim h^{k+1}(|\boldsymbol{\sigma}|_{k+1}+|u|_{k+1}){\interleave v_h\interleave}_{2,h} \;\; \forall\,v_h\in \mathring{V}_{k,m,h}^{\rm VE}.
\end{equation}
% with $0\leq m \leq 1$.
\end{lemma}
\begin{proof}
Set $\boldsymbol{w}=\div\boldsymbol{\sigma}-\nabla u$ for simplicity. Then the second equation of problem~\eqref{FSP1} becomes $\div\boldsymbol{w}=f$.

First consider the case $k=1$, that is $\mathring{V}_{1,m,h}^{\rm VE}$ with $m=1,2$ is a nonconforming finite element space.
Apply integration by parts and the weak continuity \eqref{weak-c} to acquire for $v_h\in\mathring{V}_{1,m,h}^{\rm VE}$ that
\begin{align*}
&\quad\; b_h(\boldsymbol{\sigma},v_h)-c_h(u,v_h)+(f,v_h) \\
&= \sum_{T\in \mathcal{T}_h}\sum_{F\in\mathcal{F}(T)}(\boldsymbol{n}^{\intercal}\boldsymbol{w},v_h)_{F} = \sum_{F\in\mathcal{F}_h}(\boldsymbol{n}^{\intercal}\boldsymbol{w}-Q_{0,F}(\boldsymbol{n}^{\intercal}\boldsymbol{w}), [\![v_h]\!])_{F}\\
&\leq\sum_{F\in\mathcal{F}_h}\|\boldsymbol{n}^{\intercal}\boldsymbol{w}-Q_{0,F}(\boldsymbol{n}^{\intercal}\boldsymbol{w})\|_{0,F}\|[\![v_h]\!]\|_{0,F}.
\end{align*}
By the trace inequality and the error estimate of $Q_{0,T}$,
\begin{equation}\label{Q0F}	
\|\boldsymbol{n}^{\intercal}\boldsymbol{w}-Q_{0,F}(\boldsymbol{n}^{\intercal}\boldsymbol{w})\|_{0,F}\leq \|\boldsymbol{w}-Q_{0,T}\boldsymbol{w}\|_{0,F}\lesssim h_T^{1/2}|\boldsymbol{w}|_{1,T}.
\end{equation}
Combining the last two inequalities to get
\begin{equation*}
b_h(\boldsymbol{\sigma},v_h)-c_h(u,v_h)+(f,v_h)\lesssim h|\boldsymbol{w}|_1\bigg(\sum_{F\in \mathcal{F}_h}h_F^{-1}\|[\![v_h]\!]\|_{0,F}^2\bigg)^{1/2},
\end{equation*}
which together with \eqref{cr-jump} yields \eqref{consistencyerr-1} and \eqref{consistencyerr-2} for $k=1$.

Next we consider the case $k\geq2$.	
We get from \eqref{BDMint2}-\eqref{BDMint3} that
\begin{align*}
&\quad\; b_h(\boldsymbol{\sigma},v_h)-c_h(u,v_h)+(f,Q_{m-2,h}v_h) \\
&
=(Q_h^{\rm div}\boldsymbol{w}, Q_h^{\rm div}\nabla_hv_h) + (\div\boldsymbol{w},Q_{m-2,h}v_h)=(Q_h^{\rm div}\boldsymbol{w}-I^{\div}_{h}\boldsymbol{w}, \nabla_hv_h) \\
&=\sum_{T\in \mathcal{T}_h}\sum_{F\in\mathcal{F}(T)}((Q_h^{\rm div}\boldsymbol{w}-\boldsymbol{w})\cdot\boldsymbol{n},Q_{k-1,F}v_h-Q_{m-2,T}v_h)_{F}.
\end{align*}
By the trace inequality and the error estimate \eqref{estimateQTdiv} of $Q_h^{\rm div}$, we have
\begin{align*}
&\quad\; b_h(\boldsymbol{\sigma},v_h)-c_h(u,v_h)+(f,Q_{m-2,h}v_h) \\
&\lesssim \sum_{T\in \mathcal{T}_h}\sum_{F\in\mathcal{F}(T)}h_T^{k-1/2}|\boldsymbol{w}|_{k,T}\|Q_{k-1,F}v_h-Q_{m-2,T}v_h\|_{0,F}.
\end{align*}
Finally, the estimate \eqref{consistencyerr-1} follows from \eqref{H1ineq} for $k\geq2$, and
the estimate \eqref{consistencyerr-2} from~\eqref{H2ineq} for $m\geq3$.
\end{proof}

\begin{remark}\rm
The case $k=2$ and $m=2$ is not covered by estimate \eqref{consistencyerr-2}. In this case, 
\begin{align*}
b_h(\boldsymbol{\sigma},v_h)-c_h(u,v_h)+(f,Q_{0,h}v_h) \lesssim \sum_{T\in \mathcal{T}_h}\sum_{F\in\mathcal{F}(T)}h_T^{3/2}|\boldsymbol{w}|_{2,T}\|Q_{1,F}v_h-Q_{0,T}v_h\|_{0,F}.
\end{align*}
The failure of estimate \eqref{consistencyerr-2} in this case stems from the fact that the constant projection $Q_{0,T}v_h$ lifts $\|Q_{1,F}v_h-Q_{0,T}v_h\|_{0,F}$ only to the discrete $H^1$-seminorm $|v_h|_{1,h}$ with optimal convergence rate, but not to the discrete $H^2$-seminorm ${\interleave v_h\interleave}_{2,h}$.
\end{remark}

\begin{theorem}
Let $(\boldsymbol{\sigma},u)\in H^{-1}(\div\div,\Omega; \mathbb{M})\times H_0^1(\Omega)$ and $(\boldsymbol{\sigma}_h,u_h)\in \Sigma_{r,k,m,h}^{\div}\times \mathring{V}_{k,m,h}^{\rm VE}$ be the solution of problem \eqref{FSP-mix} and the mixed method \eqref{FSP-MVEM}, respectively. Assume $u\in H^{k+3}(\Omega)$. We have
\begin{equation}
\label{errresult1}
\varepsilon^{-1}\|\boldsymbol{\sigma}-\boldsymbol{\sigma}_h\|_0+{\interleave u-u_h\interleave}_{\varepsilon,h}\lesssim h^k\|u\|_{k+3}.
\end{equation}
If $k\neq2$ and $m\neq2$,
we have
\begin{align}
\label{errresult2}
\varepsilon^{-1}\|\boldsymbol{\sigma}-\boldsymbol{\sigma}_h\|_0+{\interleave I^{\rm VE}_hu-u_h\interleave}_{\varepsilon,h}\lesssim\varepsilon h^{r+1}\|u\|_{k+3}+\varepsilon^{-1}h^{k+1}|u|_{k+1}.
\end{align}
\end{theorem}
\begin{proof}
Take $\boldsymbol{\tau}_h\in\Sigma_{r,k,m,h}^{\div}$ and $v_h\in\mathring{V}_{k,m,h}^{\rm VE}$.
From the equation \eqref{FSP-mix1} and the mixed method \eqref{FSP-MVEM}, we have the error equations
\begin{align}
\label{eq:erreqn1}
a(\Pi_{h}\boldsymbol{\sigma}-\boldsymbol{\sigma}_h,\boldsymbol{\tau}_h)+b_h(\boldsymbol{\tau}_h,u-u_h)&=a(\Pi_{h}\boldsymbol{\sigma}-\boldsymbol{\sigma},\boldsymbol{\tau}_h), \\
\label{eq:erreqn2}
b_h(\boldsymbol{\sigma}-\boldsymbol{\sigma}_h,v_h)-c_h(u-u_h,v_h) &= b_h(\boldsymbol{\sigma},v_h)-c_h(u,v_h)+\langle\!\langle f, v_h\rangle\!\rangle.
\end{align}

Noting that $b_h(\boldsymbol{\tau}_h,I^{\rm VE}_hu-u)=0$ follows from \eqref{exchange}, the error equation \eqref{eq:erreqn1} becomes
\begin{align}\label{errresult-pf1}
a(\Pi_{h}\boldsymbol{\sigma}-\boldsymbol{\sigma}_h,\boldsymbol{\tau}_h)+b_h(\boldsymbol{\tau}_h,I^{\rm VE}_hu-u_h)	&= a(\Pi_{h}\boldsymbol{\sigma}-\boldsymbol{\sigma},\boldsymbol{\tau}_h) \\
\notag
&\leq \varepsilon^{-2}\|\Pi_{h}\boldsymbol{\sigma}-\boldsymbol{\sigma}\|_0\|\boldsymbol{\tau}_h\|_0.
\end{align}
On the other side, we acquire from \eqref{div-M-int1}, \eqref{exchange} and \eqref{eq:erreqn2} that
\begin{align}\label{errresult-pf2}
b_h(\Pi_{h}\boldsymbol{\sigma}-\boldsymbol{\sigma}_h,v_h)-c_h(I^{\rm VE}_hu-u_h,v_h) 
&= b_h(\boldsymbol{\sigma}-\boldsymbol{\sigma}_h,v_h)-c_h(u-u_h,v_h)\\
\notag
&= b_h(\boldsymbol{\sigma},v_h)-c_h(u,v_h)+\langle\!\langle f, v_h\rangle\!\rangle.
\end{align}
The combination of \eqref{errresult-pf1}-\eqref{errresult-pf2} and the discrete stability \eqref{discretestability} yields
\begin{align}\label{eq:error0}
&\varepsilon^{-1}\|\Pi_{h}\boldsymbol{\sigma}-\boldsymbol{\sigma}_h\|_0+{\interleave I^{\rm VE}_hu-u_h\interleave}_{\varepsilon,h}\lesssim \varepsilon^{-1}\|\Pi_{h}\boldsymbol{\sigma}-\boldsymbol{\sigma}\|_0 \\
\notag
&\qquad\qquad\qquad\qquad\qquad\qquad + \sup_{v_h\in \mathring{V}_{k,m,h}^{\rm VE}}
\frac{b_h(\boldsymbol{\sigma},v_h)-c_h(u,v_h)+\langle\!\langle f, v_h\rangle\!\rangle}{\interleave v_h\interleave_{\varepsilon,h}}.
\end{align}
Therefore, by using the triangle inequality and \eqref{div-M-int2}, we derive \eqref{errresult1} from \eqref{consistencyerr-1} and \eqref{ve-int-err}-\eqref{exchange}, and \eqref{errresult2} from \eqref{consistencyerr-2}.
\end{proof}

\begin{remark}
When $\varepsilon\eqsim 1$, 
the estimate ${\interleave u-u_h\interleave}_{\varepsilon,h}=O(h^k)$ in \eqref{errresult1} is superconvergent for all $k\geq 1$, the estimate ${\interleave I^{\rm VE}_hu-u_h\interleave}_{\varepsilon,h}=O(h^{k+1})$ in \eqref{errresult2} is superconvergent for all $r=k\geq1$ except case $k=m=2$, and the estimate $\|\boldsymbol{\sigma}-\boldsymbol{\sigma}_h\|_0=O(h^{r+1})$ in \eqref{errresult2} is optimal for all $k\geq1$ except case $r=k=m=2$.
When $r=k=m=2$, the estimate $\|\boldsymbol{\sigma}-\boldsymbol{\sigma}_h\|_0=O(h^{2})$ in \eqref{errresult1} is suboptimal but confirmed to be sharp by numerical results.
\end{remark}

The estimates \eqref{errresult1}-\eqref{errresult2} are not robust with respect to the singular perturbation parameter $ \varepsilon $ in the presence of boundary layers. At the end of this section, we present a robust and optimal error estimate for the mixed method \eqref{FSP-MVEM}.

\begin{lemma}
Let $(\boldsymbol{\sigma},u)\in H^{-1}(\div\div,\Omega; \mathbb{M})\times H_0^1(\Omega)$ be the solution of problem~\eqref{FSP-mix}. Assume $\boldsymbol{\sigma}\in H^{1}(\Omega;\mathbb{M})$,  $u\in H^{3}(\Omega)$, and the regularity \eqref{poissonregularity} holds with $s=k+1$. We have for any $v_h\in \mathring{V}_{k,m,h}^{\rm VE}$ that
\begin{equation}
\label{consistencyerr-3}
b_h(\boldsymbol{\sigma},v_h)-c_h(u,v_h)+\langle\!\langle f, v_h\rangle\!\rangle\lesssim (\varepsilon^{1/2}\|f\|_0+h^k\|f\|_{k-1})|v_h|_{1,h}.
\end{equation}
\end{lemma}
\begin{proof}
It follows from the Cauchy-Schwarz inequality and \eqref{regularity} that
\begin{equation*}
b_h(\boldsymbol{\sigma},v_h)-c_h(u-\bar{u},v_h)\lesssim(\varepsilon^2|u|_3+|u-\bar{u}|_1)|v_h|_{1,h}\lesssim\varepsilon^{1/2}\|f\|_0|v_h|_{1,h}.
\end{equation*}
Since
\begin{equation*}
b_h(\boldsymbol{\sigma},v_h)-c_h(u,v_h)+\langle\!\langle f, v_h\rangle\!\rangle = b_h(\boldsymbol{\sigma},v_h)-c_h(u-\bar{u},v_h)-c_h(\bar{u},v_h)+\langle\!\langle f, v_h\rangle\!\rangle, 
\end{equation*} 
it suffices to prove
\begin{equation}\label{eq:20250105}
\langle\!\langle f, v_h\rangle\!\rangle-c_h(\bar{u},v_h)\lesssim h^k\|f\|_{k-1}|v_h|_{1,h}.
\end{equation}

First consider the case $k=1$. We get from \eqref{poisson}, integration by parts and the weak continuity \eqref{weak-c} that
\begin{align*}
(f, v_h) - c_h(\bar{u},v_h) =&-\sum_{T\in \mathcal{T}_h}\sum_{F\in\mathcal{F}(T)}(\partial_n\bar{u},v_h)_{F}\\
=&-\sum_{T\in \mathcal{T}_h}\sum_{F\in\mathcal{F}(T)}(\partial_n\bar{u}-Q_{0,F}(\partial_n\bar{u}),v_h-Q_{0,F}v_h)_{F} \\
\leq & \sum_{T\in \mathcal{T}_h}\sum_{F\in\mathcal{F}(T)}\|\partial_n\bar{u}-Q_{0,F}(\partial_n\bar{u})\|_{0,F}\|v_h-Q_{0,F}v_h\|_{0,F}.
\end{align*}
Similar to \eqref{Q0F}, we have
\begin{equation*}
\|\partial_n\bar{u}-Q_{0,F}(\partial_n\bar{u})\|_{0,F}\lesssim h_T^{1/2}|\bar{u}|_{2,T}, \quad \|v_h-Q_{0,F}v_h\|_{0,F} \lesssim h_T^{1/2}|v_h|_{1,T}.
\end{equation*}
The estimate \eqref{eq:20250105} holds from the last three inequalities and the regularity \eqref{poissonregularity}.

Next consider the case $k\geq2$. We get from \eqref{poisson} and \eqref{BDMint2}-\eqref{BDMint3} that
\begin{align*}
&\quad(f,Q_{m-2,h}v_h) - c_h(\bar{u},v_h) \\
&\!=\! -\big(\div(\nabla \bar{u}), Q_{m-2,h}v_h\big) - \big(Q_h^{\rm div}(\nabla \bar{u}), \nabla_hv_h\big) \!=\! \big(I^{\div}_{h}(\nabla \bar{u}) - Q_h^{\rm div}(\nabla \bar{u}), \nabla_hv_h\big) \\
&\!=\!\!\sum_{T\in \mathcal{T}_h}\!\sum_{F\in\mathcal{F}(T)}\!\big((\nabla\bar{u}-Q_h^{\rm div}(\nabla\bar{u}))\cdot\boldsymbol{n},Q_{k-1,F}v_h-Q_{m-2,T}v_h\big)_{F}.
\end{align*}
On the other side, by the trace inequality and the error estimate \eqref{estimateQTdiv} of $Q_h^{\rm div}$,
\begin{align*}
\|\nabla\bar{u}-Q_h^{\rm div}(\nabla\bar{u})\|_{0,F}&\lesssim h_T^{-1/2}\|\nabla\bar{u}-Q_h^{\rm div}(\nabla\bar{u})\|_{0,T} + h_T^{1/2}|\nabla\bar{u}-Q_h^{\rm div}(\nabla\bar{u})|_{1,T} \\
&\lesssim h_T^{k-1/2}|\bar{u}|_{k+1,T}.
\end{align*}
Finally, we conclude \eqref{eq:20250105} from the regularity \eqref{poissonregularity} and~\eqref{H1ineq}.
\end{proof}

\begin{theorem}
Let $(\boldsymbol{\sigma},u)\in H^{-1}(\div\div,\Omega; \mathbb{M})\times H_0^1(\Omega)$ and $(\boldsymbol{\sigma}_h,u_h)\in \Sigma_{r,k,m,h}^{\div}\times \mathring{V}_{k,m,h}^{\rm VE}$ be the solution of problem \eqref{FSP-mix} and the mixed method \eqref{FSP-MVEM}, respectively. Assume $\boldsymbol{\sigma}\in H^{1}(\Omega;\mathbb{M})$,  $u\in H^{3}(\Omega)$, and the regularity \eqref{poissonregularity} holds with $s=k+1$. We have
\begin{align}
\label{errresult3}
\varepsilon^{-1}\|\boldsymbol{\sigma}-\boldsymbol{\sigma}_h\|_0+{\interleave u-u_h\interleave}_{\varepsilon,h}&\lesssim \varepsilon^{1/2}\|f\|_{0}+h^{k}\|f\|_{k-1},\\
\label{errresult4}
\varepsilon|Q_h^{\rm div}\nabla_h(\bar{u}-u_h)|_{1,h}+|\bar{u}-u_h|_{1,h}&\lesssim \varepsilon^{1/2}\|f\|_{0}+h^{k}\|f\|_{k-1}.
\end{align}
\end{theorem}
\begin{proof}
Employ the estimate \eqref{div-M-int2} and the regularity \eqref{regularity} to have
\begin{equation}
\label{errresult-pf4}
\varepsilon^{-1}\|\boldsymbol{\sigma}-\Pi_{h}\boldsymbol{\sigma}\|_{0}\lesssim \varepsilon^{-1}\|\boldsymbol{\sigma}\|_0\lesssim\varepsilon|u|_2\lesssim\varepsilon^{1/2}\|f\|_0.
\end{equation}
By substituting \eqref{errresult-pf4} and \eqref{consistencyerr-3} into \eqref{eq:error0}, we acquire
\begin{equation}\label{eq:err20250106}
\varepsilon^{-1}\|\Pi_{h}\boldsymbol{\sigma}-\boldsymbol{\sigma}_h\|_0+{\interleave I^{\rm VE}_hu-u_h\interleave}_{\varepsilon,h}\lesssim \varepsilon^{1/2}\|f\|_0+h^k\|f\|_{k-1}.
\end{equation}
On the other side, it follows from \eqref{ve-int-err} and \eqref{poissonregularity}-\eqref{regularity} that
\begin{equation}
\label{errresult-pf3}
\begin{aligned}
|u-I^{\rm VE}_hu|_{1,h}&\lesssim |(u-\bar{u})-I^{\rm VE}_h(u-\bar{u})|_{1,h}+|\bar{u}-I^{\rm VE}_h\bar{u}|_{1,h}\\
% \notag
&\lesssim |u-\bar{u}|_{1}+h^k|\bar{u}|_{k+1}\lesssim\varepsilon^{1/2}\|f\|_{0}+h^{k}\|f\|_{k-1}.
\end{aligned}
\end{equation}
Noting that ${\interleave u-I^{\rm VE}_hu\interleave}_{2,h} = 0$ by \eqref{exchange}, apply the triangle inequality to have
\begin{align*}
\varepsilon^{-1}\|\boldsymbol{\sigma}-\boldsymbol{\sigma}_h\|_0+{\interleave u-u_h\interleave}_{\varepsilon,h}&\leq  \varepsilon^{-1}(\|\boldsymbol{\sigma}-\Pi_{h}\boldsymbol{\sigma}\|_0+\|\Pi_{h}\boldsymbol{\sigma}-\boldsymbol{\sigma}_h\|_0)\\
&\quad+{\interleave I^{\rm VE}_hu-u_h\interleave}_{\varepsilon,h}+|u-I^{\rm VE}_hu|_{1,h},
\end{align*}
which together with \eqref{errresult-pf4}-\eqref{errresult-pf3} yields \eqref{errresult3}.

Next we prove \eqref{errresult4}. By \eqref{poissonregularity}-\eqref{regularity} and the error estimate \eqref{estimateQTdiv} of $Q_h^{\rm div}$, we get
\begin{align*}
\varepsilon|Q_h^{\rm div}\nabla(\bar{u}-u)|_{1,h}+|\bar{u}-u|_{1}\lesssim \varepsilon|\bar{u}-u|_{2}+|\bar{u}-u|_{1}\lesssim (\varepsilon+\varepsilon^{1/2})\|f\|_0\lesssim\varepsilon^{1/2}\|f\|_{0}.
\end{align*}
Thus, the estimate \eqref{errresult4} follows from the triangle inequality and \eqref{errresult3}. 
\end{proof}

%%%%%%%%%%%%%%%%%%%%%%%%%%%%%%%%%%%%%%%%
%% VEM method
%%%%%%%%%%%%%%%%%%%%%%%%%%%%%%%%%%%%%%%%
%\section{Hybridization}\label{sec5}
\section{Connection to Other Methods}\label{sec5}
In this section, we will hybridize the mixed method \eqref{FSP-MVEM}, and connect it to weak Galerkin methods and a mixed finite element method for problem \eqref{FSP0}.

\subsection{Stabilization-free weak Galerkin methods}
To connect the mixed method~\eqref{FSP-MVEM} to weak Galerkin methods, introduce the following discrete broken spaces
\begin{align*}
M_{m-2, k-1}^{-1}&:=\mathbb P_{m-2}(\mathcal{T}_h)\times \mathbb P_{k-1}(\mathring{\mathcal{F}}_h), \\
M_{m-2, k-1,r}^{-1}&:=\mathbb P_{m-2}(\mathcal{T}_h)\times \mathbb P_{k-1}(\mathring{\mathcal{F}}_h)\times \mathbb P_{r}(\mathring{\mathcal{F}}_h;\mathbb R^d), \\
V_{k-1,m-1}^{-1}&:=\{\boldsymbol{v}\in L^{2}(\Omega;\mathbb{R}^d):\boldsymbol{v}|_T\in V_{k-1,m-1}(T;\mathbb R^d)~\textrm{for}~T\in\mathcal{T}_h\}, \\
\Sigma_{r,k,m}^{-1}&:=\{\boldsymbol{\tau}\in L^{2}(\Omega;\mathbb{M}):\boldsymbol{\tau}|_T\in \Sigma_{r,k,m}(T;\mathbb{M})~\textrm{for}~T\in\mathcal{T}_h\}.
\end{align*}

We first recast the mixed method \eqref{FSP-MVEM} with the help of the weak gradient.
Define the interpolation operator $I_h^{\rm NC}: L^2(\mathcal{T}_h)\times L^2(\mathring{\mathcal{F}}_h)\to \mathring{V}_{k,m,h}^{\rm VE}$ as follows: for $v=(v_0, v_b)\in L^2(\mathcal{T}_h)\times L^2(\mathring{\mathcal{F}}_h)$, the function $I_h^{\rm NC}v\in \mathring{V}_{k,m,h}^{\rm VE}$ is determined by
\begin{align*}
Q_{k-1,F}(I_h^{\rm NC}v)&=Q_{k-1,F}v_b \qquad\, \forall\,F\in\mathring{\mathcal{F}}_h, \\
Q_{m-2,T}(I_h^{\rm NC}v)&=Q_{m-2,T}v_0 \qquad \forall\,T\in\mathcal{T}_h. 
\end{align*}

\begin{lemma}\label{lem:IhNCbijection}
The interpolation operator $I_h^{\rm NC}: M_{m-2, k-1}^{-1}\to \mathring{V}_{k,m,h}^{\rm VE}$ is bijective.
\end{lemma}
\begin{proof}
For $v=(v_0, v_b)\in M_{m-2, k-1}^{-1}$, by the definition of $I_h^{\rm NC}v$, we have
\begin{align*}
Q_{k-1,F}(I_h^{\rm NC}v)=v_b \;\; \textrm{ for } F\in\mathring{\mathcal{F}}_h;\quad
Q_{m-2,T}(I_h^{\rm NC}v)=v_0 \;\; \textrm{ for } T\in\mathcal{T}_h. 
\end{align*}
This shows that $I_h^{\rm NC}$ restricted to $M_{m-2, k-1}^{-1}$ is injective. Then we conclude the result from $\dim M_{m-2, k-1}^{-1}=\dim\mathring{V}_{k,m,h}^{\rm VE}$.
\end{proof}

Define the weak gradient $\nabla_w:M_{m-2, k-1}^{-1}\to V_{k-1,m-1}^{-1}$ as follows: for $v=(v_0, v_b)\in M_{m-2, k-1}^{-1}$, let $\nabla_wv\in V_{k-1,m-1}^{-1}$ be determined by
\begin{equation*}
(\nabla_wv, \boldsymbol{q})_T=-(v_0,\div\boldsymbol{q})_T+(v_b, \boldsymbol{q}\cdot\boldsymbol{n})_{\partial T}\quad\forall\,\boldsymbol{q}\in V_{k-1,m-1}(T;\mathbb R^d), T\in\mathcal{T}_h.
\end{equation*}

\begin{lemma}
It holds
\begin{equation}\label{eq:weakgradprop}
\nabla_wv=Q_h^{\rm div}\nabla_h(I_h^{\rm NC}v)\quad\forall\,v\in M_{m-2, k-1}^{-1}.
\end{equation}
\end{lemma}
\begin{proof}
From integration by parts and the definitions of $\nabla_wv$ and $I_h^{\rm NC}v$, 
\begin{align*}
(\nabla_wv-\nabla(I_h^{\rm NC}v), \boldsymbol{q})_T= (I_h^{\rm NC}v-v_0,\div\boldsymbol{q})_T-(I_h^{\rm NC}v-v_b, \boldsymbol{q}\cdot\boldsymbol{n})_{\partial T}=0
\end{align*}
for $\boldsymbol{q}\in V_{k-1,m-1}(T;\mathbb R^d)$ and $T\in\mathcal{T}_h$. So \eqref{eq:weakgradprop} is true.
\end{proof}

By combining \eqref{eq:weakgradprop}, the norm equivalence \eqref{ve-infsup2} and Lemma~\ref{lem:IhNCbijection}, we conclude that $\|\nabla_w v\|_0$ defines a norm on the space $M_{m-2, k-1}^{-1}$.

Employing Lemma~\ref{lem:IhNCbijection} and \eqref{eq:weakgradprop}, the mixed method \eqref{FSP-MVEM} can be recast as the following stablization-free weak Galerkin method:
find $\boldsymbol{\sigma}_h\in 	\Sigma_{r,k,m,h}^{\div}$ and $u_h\in M_{m-2, k-1}^{-1}$ such that
\begin{subequations}\label{MgradWG}
\begin{align}
\label{MgradWG1}
a(\boldsymbol{\sigma}_h,\boldsymbol{\tau}_h)+(\div\boldsymbol{\tau}_h,\nabla_wu_h) &= 0 \qquad\qquad\quad\;\;\;\;\,\! \forall\, \boldsymbol{\tau}_h\in  \Sigma_{r,k,m,h}^{\div}, \\
\label{MgradWG2}
(\div\boldsymbol{\sigma}_h,\nabla_wv_h)-(\nabla_wu_h,\nabla_wv_h)&= -\langle\!\langle f, I_h^{\rm NC}v_h\rangle\!\rangle \quad \forall\, v_h\in M_{m-2, k-1}^{-1}.
\end{align}
\end{subequations}
% When $k\geq2$, equation \eqref{MgradWG2} is equivalent to $\div\boldsymbol{\sigma}_h-\nabla_wu_h\in H(\div,\Omega)$, and
% \begin{equation}\label{MgradWG20}
% (\div(\div\boldsymbol{\sigma}_h-\nabla_wu_h), v_0)=(f, v_0)\quad \forall\,v_0\in\mathbb P_{m-2}(\mathcal{T}_h).
% \end{equation}

The weak Galerkin method \eqref{MgradWG} can be further hybridized by enforcing the normal continuity of the space $\Sigma_{r,k,m,h}^{\div}$ weakly through a Lagrangian multiplier.
To this end, 
define the weak Hessian $\nabla^2_w:M_{m-2, k-1,r}^{-1}\to \Sigma_{r,k,m}^{-1}$ as follows: for $v=(v_0, v_b, \boldsymbol{v}_g)\in M_{m-2, k-1,r}^{-1}$, let $\nabla^2_wv\in \Sigma_{r,k,m}^{-1}$ be determined by
\begin{align*}
(\nabla^2_wv, \boldsymbol{\tau})_T&=-(\nabla_w(v_0,v_b),\div\boldsymbol{\tau})_T+(\boldsymbol{v}_g, \boldsymbol{\tau}\boldsymbol{n})_{\partial T} \\
&=(v_0, \div\div\boldsymbol{\tau})_T-(v_b,\boldsymbol{n}^{\intercal}\div\boldsymbol{\tau})_T+(\boldsymbol{v}_g, \boldsymbol{\tau}\boldsymbol{n})_{\partial T}
\end{align*}
for $\boldsymbol{\tau}\in \Sigma_{r,k,m}(T;\mathbb{M})$ and $T\in\mathcal{T}_h$.

\begin{lemma}
It holds the norm equivalence
\begin{equation}\label{eq:weakHessiannormequiv}
\|\nabla^2_wv\|_0^2\eqsim |\nabla_w(v_0,v_b)|_{1,h}^2+\sum_{T\in\mathcal{T}_h}h_T^{-1}\|\nabla_w(v_0,v_b)-\boldsymbol{v}_g\|_{0,\partial T}^2
\end{equation}
for $v=(v_0, v_b, \boldsymbol{v}_g)\in M_{m-2, k-1,r}^{-1}$.
\end{lemma}
\begin{proof}
Set $\boldsymbol{w}=\nabla_w(v_0,v_b)\in V_{k-1,m-1}^{-1}$ for simplicity.
Apply integration by parts to get
\begin{equation}\label{eq:20250107}
(\nabla^2_wv, \boldsymbol{\tau})_T=(\grad\boldsymbol{w}, \boldsymbol{\tau})_T+(\boldsymbol{v}_g-\boldsymbol{w}, \boldsymbol{\tau}\boldsymbol{n})_{\partial T}\quad\forall\,\boldsymbol{\tau}\in \Sigma_{r,k,m}(T;\mathbb{M}), T\in\mathcal{T}_h.
\end{equation}
By choosing $\boldsymbol{\tau}=(\nabla^2_wv)|_T$ in \eqref{eq:20250107}, we have
\begin{equation*}
\|\nabla^2_wv\|_{0,T}^2\leq |\boldsymbol{w}|_{1,T}\|\nabla^2_wv\|_{0,T} + \|\boldsymbol{w}-\boldsymbol{v}_g\|_{0,\partial T}\|\nabla^2_wv\|_{0,\partial T}.
\end{equation*}
On the other hand, by the trace inequality and the inverse inequality,
$$
\|\nabla^2_wv\|_{0,\partial T}\lesssim h_T^{-1/2}\|\nabla^2_wv\|_{0,T} + h_T^{1/2}|\nabla^2_wv|_{1,T}\lesssim h_T^{-1/2}\|\nabla^2_wv\|_{0,T},
$$
hence it follows that
\begin{equation}\label{eq:202501071}
\|\nabla^2_wv\|_{0,T}\lesssim |\boldsymbol{w}|_{1,T}+h_T^{-1/2}\|\boldsymbol{w}-\boldsymbol{v}_g\|_{0,\partial T}.
\end{equation}

For the other side, let $\boldsymbol{\tau}\in\Sigma_{r,k,m}(T;\mathbb{M})$ such that all the DoFs \eqref{div-M-dof} vanish except 
\begin{align*}
(\boldsymbol{\tau}\boldsymbol{n}, \boldsymbol{q})_F&=h_T^{-1}(\boldsymbol{v}_g-\boldsymbol{w}, \boldsymbol{q})_F \quad\quad \forall\, \boldsymbol{q}\in\mathbb{P}_{r}(F;\mathbb{R}^d), F\in\mathcal F(T), \\
(\boldsymbol{\tau},\boldsymbol{q})_T &= (\grad\boldsymbol{w},\boldsymbol{q})_T \quad\qquad\;\;\; \forall\, \boldsymbol{q}\in\grad(V_{k-1,m-1}(T;\mathbb R^d)).
\end{align*}
By a scaling argument and \eqref{eq:20250107}, we have
\begin{align*}
\|\boldsymbol{\tau}\|_{0,T}&\lesssim |\boldsymbol{w}|_{1,T} + h_T^{-1/2}\|\boldsymbol{w}-\boldsymbol{v}_g\|_{0,\partial T}, \\
(\nabla^2_wv, \boldsymbol{\tau})_T&=|\boldsymbol{w}|_{1,T}^2 + h_T^{-1}\|\boldsymbol{w}-\boldsymbol{v}_g\|_{0,\partial T}^2.
\end{align*}
Hence, 
\begin{align*}
|\boldsymbol{w}|_{1,T}^2 + h_T^{-1}\|\boldsymbol{w}-\boldsymbol{v}_g\|_{0,\partial T}^2&=(\nabla^2_wv, \boldsymbol{\tau})_T\leq \|\nabla^2_wv\|_{0,T}\|\boldsymbol{\tau}\|_{0,T} \\
&\lesssim \|\nabla^2_wv\|_{0,T}(|\boldsymbol{w}|_{1,T} + h_T^{-1/2}\|\boldsymbol{w}-\boldsymbol{v}_g\|_{0,\partial T}),
\end{align*}
which yields 
\begin{equation*}
|\boldsymbol{w}|_{1,T} + h_T^{-1/2}\|\boldsymbol{w}-\boldsymbol{v}_g\|_{0,\partial T}\lesssim \|\nabla^2_wv\|_{0,T}.
\end{equation*}
This together with \eqref{eq:202501071} implies \eqref{eq:weakHessiannormequiv}.
\end{proof}

An immediate consequence of the norm equivalence \eqref{eq:weakHessiannormequiv} is that $\|\nabla_w^2 v\|_0$ defines a norm on the space $M_{m-2, k-1, r}^{-1}$, and $\nabla^2_w:M_{m-2, k-1,r}^{-1}\to \Sigma_{r,k,m}^{-1}$ is injective.

Now we propose a fully weak Galerkin method for the fourth-order elliptic singular perturbation problem~\eqref{FSP1}:
find $u_h\in M_{m-2, k-1,r}^{-1}$ such that
\begin{equation}\label{MhessWG}
\varepsilon^2(\nabla^2_wu_h, \nabla^2_wv_h)+(\nabla_w(u_0,u_b),\nabla_w(v_0,v_b))= \langle\!\langle f, I_h^{\rm NC}(v_0,v_b)\rangle\!\rangle
\end{equation}
for any $v_h\in M_{m-2, k-1,r}^{-1}$.
The weak Galerkin method \eqref{MhessWG} is well-proposed.

\begin{theorem}
Let $u_h=(u_0,u_b,\boldsymbol{u}_g)\in M_{m-2, k-1,r}^{-1}$ be the solution of the weak Galerkin method~\eqref{MhessWG}.
Set $\boldsymbol{\sigma}_h=\varepsilon^2\nabla^2_wu_h$. Then $\boldsymbol{\sigma}_h\in \Sigma_{r,k,m,h}^{\div}$ and $(u_0,u_b)\in M_{m-2, k-1}^{-1}$ satisfy the weak Galerkin method~\eqref{MgradWG}.
Consequently, the weak Galerkin method~\eqref{MhessWG} is equivalent to the mixed method \eqref{FSP-MVEM}.
\end{theorem}
\begin{proof}
By choosing $v_h=(0,0,\boldsymbol{v}_g)\in M_{m-2, k-1,r}^{-1}$ in \eqref{MhessWG}, we get $\boldsymbol{\sigma}_h\in \Sigma_{r,k,m,h}^{\div}$. By the definition of $\boldsymbol{\sigma}_h$,
\begin{equation*}
a(\boldsymbol{\sigma}_h,\boldsymbol{\tau}_h)=(\nabla^2_wu_h, \boldsymbol{\tau}_h)=-(\nabla_w(u_0,u_b), \div\boldsymbol{\tau}_h)\quad\forall\,\boldsymbol{\tau}_h\in \Sigma_{r,k,m,h}^{\div}.
\end{equation*}
This is exactly \eqref{MgradWG1}.

For $v_h=(v_0,v_b,0)\in M_{m-2, k-1,r}^{-1}$, by the definition of $\nabla^2_wv_h$,
\begin{equation*}
(\boldsymbol{\sigma}_h, \nabla^2_wv_h)=-(\div\boldsymbol{\sigma}_h, \nabla_w(v_0,v_b)).
\end{equation*}
Now taking $v_h=(v_0,v_b,0)\in M_{m-2, k-1,r}^{-1}$ in \eqref{MhessWG} will induce \eqref{MgradWG2}. 
\end{proof}

\subsection{Mixed finite element method}
Based on the mixed formulation \eqref{FSP1st-mix}, we propose the following mixed finite element method: find $(\boldsymbol{\sigma}_h,\boldsymbol{\phi}_h,\boldsymbol{p}_h, u_h)\in \Sigma_{r,k,m,h}^{\div}\times V_{k-1,m-1,h}^{\div}\times V_{k-1,m-1}^{-1}\times \mathbb P_{m-2}(\mathcal{T}_h)$ such that
\begin{subequations}\label{FSP1st-mfem}
\begin{align}
\label{FSP1st-mfem1}
\varepsilon^{-2}(\boldsymbol{\sigma}_h,\boldsymbol{\tau}_h)+\bar{b}(\boldsymbol{\tau}_h,\boldsymbol{\psi}_h;\boldsymbol{p}_h, u_h) &= 0, \\
\label{FSP1st-mfem2}
\bar{b}(\boldsymbol{\sigma}_h,\boldsymbol{\phi}_h;\boldsymbol{q}_h, v_h)-(\boldsymbol{p}_h,\boldsymbol{q}_h) &= -(f,v_h)
\end{align}
for $\boldsymbol{\tau}_h\in \Sigma_{r,k,m,h}^{\div}$, $\boldsymbol{\psi}_h\in V_{k-1,m-1,h}^{\div}$, $\boldsymbol{q}_h\in V_{k-1,m-1}^{-1}$ and $v_h\in \mathbb P_{m-2}(\mathcal{T}_h)$.
\end{subequations}

\begin{lemma}
The mixed finite element method \eqref{FSP1st-mfem} is well-posed.
\end{lemma}
\begin{proof}
It suffices to prove that the mixed method \eqref{FSP1st-mfem} admits only the zero solution when $ f = 0 $. By selecting $\boldsymbol{\tau}_h = \boldsymbol{\sigma}_h$, $\boldsymbol{\psi}_h = \boldsymbol{\phi}_h$, $\boldsymbol{q}_h = \boldsymbol{p}_h$, and $v_h = u_h$ in \eqref{FSP1st-mfem}, subtracting \eqref{FSP1st-mfem2} from \eqref{FSP1st-mfem1} leads to $\boldsymbol{\sigma}_h = 0$ and $\boldsymbol{p}_h = 0$.
Then $u_h=0$ follows from \eqref{FSP1st-mfem1} and $\div V_{k-1,m-1,h}^{\div}=\mathbb P_{m-2}(\mathcal{T}_h)$, and $\boldsymbol{\phi}_h=0$ follows from \eqref{FSP1st-mfem2} and $V_{k-1,m-1,h}^{\div}\subseteq V_{k-1,m-1}^{-1}$.
\end{proof}

\begin{theorem}
Let $(\boldsymbol{\sigma}_h,u_h)\in \Sigma_{r,k,m,h}^{\div}\times \mathring{V}_{k,m,h}^{\rm VE}$ with $k\geq1$ satisfy
\begin{subequations}\label{FSP-MVEMnew}
\begin{align}
\label{FSP-MVEMnew1}
a(\boldsymbol{\sigma}_h,\boldsymbol{\tau}_h)+b_h(\boldsymbol{\tau}_h,u_h) &= 0 \qquad\qquad\qquad\;\;\,\;\, \forall\, \boldsymbol{\tau}_h\in  \Sigma_{r,k,m,h}^{\div}, \\
\label{FSP-MVEMnew2}
b_h(\boldsymbol{\sigma}_h,v_h)-c_h(u_h,v_h)&= -(f, Q_{m-2,h}v_h) \quad \forall\, v_h\in \mathring{V}_{k,m,h}^{\rm VE},
\end{align}
\end{subequations} 
where the bilinear forms $b_h(\cdot,\cdot)$ and $c_h(\cdot,\cdot)$ are defined in \eqref{bhch}.
Then $(\boldsymbol{\sigma}_h,\div\boldsymbol{\sigma}_h-Q_h^{\rm div}\nabla_h u_h,Q_h^{\rm div}\nabla_h u_h, Q_{m-2,h}u_h)\in \Sigma_{r,k,m,h}^{\div}\times V_{k-1,m-1,h}^{\div}\times V_{k-1,m-1}^{-1}\times \mathbb P_{m-2}(\mathcal{T}_h)$ is the solution of the mixed finite element method \eqref{FSP1st-mfem}.
Moreover, the mixed finite element method \eqref{FSP1st-mfem} is equivalent to the mixed method \eqref{FSP-MVEM} for $k\geq2$.
\end{theorem}
\begin{proof}
Since the mixed methods \eqref{FSP-MVEMnew} and \eqref{FSP-MVEM} differ only in the right-hand side, the mixed method \eqref{FSP-MVEMnew} is well-posed. Following the proof of Theorem~\ref{thm:FSP-MVEM}, we have $\div\boldsymbol{\sigma}_h-Q_h^{\rm div}\nabla_h u_h\in H(\div,\Omega)$ for $k\geq1$.

For $\boldsymbol{\tau}_h\in \Sigma_{r,k,m,h}^{\div}$ and $\boldsymbol{\psi}_h\in V_{k-1,m-1,h}^{\div}$,
we get from integration by parts and~\eqref{FSP-MVEMnew1} that
\begin{align*}
&\quad\; \varepsilon^{-2}(\boldsymbol{\sigma}_h,\boldsymbol{\tau}_h)+\bar{b}(\boldsymbol{\tau}_h,\boldsymbol{\psi}_h;Q_h^{\rm div}\nabla_h u_h, Q_{m-2,h}u_h) \\
&= \varepsilon^{-2}(\boldsymbol{\sigma}_h,\boldsymbol{\tau}_h)+(\div\boldsymbol{\tau}_h-\boldsymbol{\psi}_h, \nabla_h u_h)-(\div\boldsymbol{\psi}_h, u_h) \\
&= \varepsilon^{-2}(\boldsymbol{\sigma}_h,\boldsymbol{\tau}_h)+(\div\boldsymbol{\tau}_h, \nabla_h u_h)=0.
\end{align*}
So \eqref{FSP1st-mfem1} is true.
For $\boldsymbol{q}_h\in V_{k-1,m-1}^{-1}$ and $v_h\in \mathring{V}_{k,m,h}^{\rm VE}$, applying integration by parts again, we have from \eqref{FSP-MVEMnew2} and the fact $\div\boldsymbol{\sigma}_h-Q_h^{\rm div}\nabla_h u_h\in H(\div,\Omega)$ that
\begin{align*}
&\quad \bar{b}(\boldsymbol{\sigma}_h,\div\boldsymbol{\sigma}_h-Q_h^{\rm div}\nabla_h u_h;\boldsymbol{q}_h, Q_{m-2,h}v_h)-(Q_h^{\rm div}\nabla_h u_h,\boldsymbol{q}_h) \\
&=-(\div(\div\boldsymbol{\sigma}_h-Q_h^{\rm div}\nabla_h u_h), v_h) \\
&=(\div\boldsymbol{\sigma}_h-Q_h^{\rm div}\nabla_h u_h, \nabla_hv_h) = -(f, Q_{m-2,h}v_h).
\end{align*}
Thus, we conclude \eqref{FSP1st-mfem2} from the fact $Q_{m-2,h}\mathring{V}_{k,m,h}^{\rm VE}=\mathbb P_{m-2}(\mathcal{T}_h)$.

Finally, the equivalence between the mixed finite element method \eqref{FSP1st-mfem} and the mixed method \eqref{FSP-MVEM} for $k\geq 2$ follows from the fact that the mixed methods \eqref{FSP-MVEMnew} and \eqref{FSP-MVEM} coincide exactly for $k\geq 2$.
\end{proof}

\section{Numerical Results}\label{sec6}
In this section, we will numerically examine the performance of the mixed method~\eqref{FSP-MVEM}.
Let $\Omega$ be the unit square $(0, 1)^2$.
All the numerical tests are performed on the uniform triangulation.

\begin{example}\label{example1}
\normalfont
We first test the discrete method \eqref{FSP-MVEM} with the exact solution
$$u = \sin^2(\pi x)\sin^2(\pi y).$$
The right-hand side $f$ is computed from \eqref{FSP0}. Notice that the solution $u$ does not have boundary layers.
	
We measure the numerical error
\begin{align*}
\textrm{Err}_1 :=& \varepsilon^{-1}\|\boldsymbol{\sigma}-\boldsymbol{\sigma}_h\|_{0}+(\varepsilon^2{\interleave u-u_h\interleave}^2_{2,h}+\|\nabla u-Q_h^{\rm div}\nabla_hu_h\|^2_{0})^{1/2} %\\
\end{align*}
with $r = m = k$.
The numerical error $\textrm{Err}_1$ with different $\varepsilon$ , $h$ and $k$ is shown in Table~\ref{ex1-k123}.
We observe from Table~\ref{ex1-k123} that $\textrm{Err}_1\eqsim O(h^{k})$ for $\varepsilon = 1, 10^{-1}, 10^{-5}, 10^{-6}$, which is optimal and consistent with \eqref{errresult1}.
	
\begin{table}[ht]
	\centering
	\caption{$\rm Err_1$ of the discrete method \eqref{FSP-MVEM} for Example \ref{example1} with $r=m=k$.}
	\label{ex1-k123}
	\begin{tabular}{cccc|cc|cc|cc}
		\toprule
		\multirow{2}{*}{$k$}& \multirow{2}{*}{$h$} & \multicolumn{8}{c}{$\varepsilon$} \\
		\cmidrule(lr){3-10}
		& & $1$ & rate & $10^{-1}$ & rate & $10^{-5}$ & rate & $10^{-6}$ & rate\\
		\midrule
		\multirow{5}{*}{1} 
		& $1/16$ & 4.233e-01 &  & 2.354e-01 &  & 2.611e-01 & & 2.610e-01 & \\
		& $1/32$ & 1.552e-01 & 1.45 & 1.104e-01 & 1.09 & 1.311e-01 & 0.99 & 1.309e-01 & 1.00 \\
		& $1/64$ & 6.418e-02 & 1.27 & 5.331e-02 & 1.05 &  6.567e-02 & 1.00 & 6.553e-02 & 1.00 \\
		& $1/128$ & 2.885e-02 & 1.15 & 2.618e-02 & 1.03 &  3.292e-02 & 1.00 & 3.278e-02 & 1.00 \\
		& $1/256$ & 1.363e-02 & 1.08 & 1.297e-02 & 1.01 &  1.654e-02 & 0.99 & 1.640e-02 & 1.00\\
		\midrule
		\multirow{5}{*}{2} 
		& $1/8$ & 1.453e+00 &   & 1.974e-01 &   & 1.124e-01 & & 1.124e-01 & \\
		& $1/16$ & 3.761e-01 & 1.95 & 5.101e-02 & 1.95  & 2.912e-02 & 1.95 & 2.911e-02 & 1.95 \\
		& $1/32$ & 9.487e-02 & 1.99 & 1.287e-02 & 1.99 & 7.360e-03 & 1.98 & 7.354e-03 & 1.98\\
		& $1/64$ & 2.377e-02 & 2.00 & 3.224e-03 & 2.00 & 1.848e-03 & 1.99 & 1.845e-03 & 1.99\\
		& $1/128$ & 5.952e-03 & 2.00 & 8.065e-04 & 2.00 & 4.633e-04 & 2.00 & 4.618e-04 & 2.00\\
		\midrule
		\multirow{5}{*}{3} 
		& $1/4$ & 5.079e-01 &   & 8.470e-02 &   & 5.902e-02 & & 5.900e-02 & \\
		& $1/8$ & 3.717e-02 & 3.77  & 8.278e-03 & 3.36 & 7.926e-03 & 2.90 & 7.922e-03 & 2.90 \\
		& $1/16$ & 2.540e-03 & 3.87  & 8.392e-04 & 3.30 & 1.017e-03 & 2.96 & 1.016e-03 & 2.96 \\
		& $1/32$ & 1.860e-04 & 3.77  & 9.284e-05 & 3.18 & 1.285e-04 & 2.99 & 1.282e-04 & 2.99 \\
		& $1/64$ & 1.597e-05 & 3.54  & 1.089e-05 & 3.09 & 1.615e-05 & 2.99 & 1.609e-05 & 2.99\\
		\bottomrule
	\end{tabular}
\end{table}
\end{example}

\begin{example}\label{example2}
\normalfont
This example is designed to verify the estimate \eqref{errresult2}; therefore, only the cases $\varepsilon = 1, 10^{-1}$ are considered.
Take the same $u$ as in Example \ref{example1}. 
By \eqref{ve-infsup2} and \eqref{exchange}, it can be seen that
\begin{equation*}
	\|\nabla_h(I^{\rm VE}_hu)-\nabla_h u_h\|_0\eqsim \|Q_h^{\rm div}(\nabla_h(I^{\rm VE}_hu)-\nabla_h u_h)\|_0 = \|Q_h^{\rm div}(\nabla u-\nabla_h u_h)\|_{0}.
\end{equation*}
Then we compute the numerical errors
\begin{align*}
	\textrm{Err}_u :=& (\varepsilon^2{\interleave u-u_h\interleave}^2_{2,h}+\|Q_h^{\rm div}(\nabla u-\nabla_h u_h)\|^2_{0})^{1/2}\eqsim{\interleave I^{\rm VE}_hu-u_h\interleave}_{\varepsilon,h},\\
	\textrm{Err}_{\boldsymbol{\sigma}}:=&\varepsilon^{-1}\|\boldsymbol{\sigma}-\boldsymbol{\sigma}_h\|_0.
\end{align*}
As observed from Tables \ref{ex2-k123}-\ref{ex2-k2}, in the cases of $r=m=k=1,3$ and $r=k=2, m=3$, $\textrm{Err}_u\eqsim O(h^{k+1})$ demonstrates superconvergence, while $\textrm{Err}_{\boldsymbol{\sigma}}\eqsim O(h^{k+1})$ exhibits optimal convergence. These results are consistent with the estimate \eqref{errresult2}. However, for $r=m=k=2$, it can be seen from Table \ref{ex2-k123} that $\textrm{Err}_u\eqsim O(h^{2})$ is optimal, whereas $\textrm{Err}_{\boldsymbol{\sigma}}\eqsim O(h^{2})$ is only suboptimal. This indicates that, in this case, $\textrm{Err}_u$ and $\textrm{Err}_{\boldsymbol{\sigma}}$ satisfy \eqref{errresult1} but do not satisfy \eqref{errresult2}.
\begin{table}[h]
	\centering
	\caption{The performance of the discrete method \eqref{FSP-MVEM} for Example \ref{example2} with $r=m=k$.}
	\label{ex2-k123}
	\begin{tabular}{cccc|cc|cc|cc}
		\toprule
		\multirow{2}{*}{$k$}& \multirow{2}{*}{$h$} & \multicolumn{4}{c}{$\varepsilon=1$} & \multicolumn{4}{c}{$\varepsilon=10^{-1}$} \\
		\cmidrule(lr){3-6}\cmidrule(lr){7-10}
		& & $ \rm Err_{\boldsymbol{\sigma}} $ & rate & $\rm Err_u$ & rate & $ \rm Err_{\boldsymbol{\sigma}} $ & rate & $\rm Err_u$ & rate\\
		\midrule
		\multirow{5}{*}{1} 
		& $1/16$ & 1.959e-01 & & 9.855e-02 & & 2.989e-02 & & 1.570e-02 & \\
		& $1/32$ & 4.937e-02 & 1.99 & 2.551e-02 & 1.95 & 7.587e-03 & 1.98 & 3.988e-03 & 1.98 \\
		& $1/64$ & 1.238e-02 & 2.00 & 6.454e-03 & 1.98 & 1.907e-03 & 1.99 & 1.004e-03 & 1.99\\
		& $1/128$ & 3.100e-03 & 2.00 & 1.621e-03 & 2.00 & 4.778e-04 & 2.00 & 2.520e-04 & 1.99\\
		& $1/256$ & 7.756e-04 & 2.00 & 4.059e-04 & 2.00 & 1.196e-04 & 2.00 & 6.311e-05 & 2.00\\
		\midrule
		\multirow{5}{*}{2} 
		& $1/8$ & 7.374e-01 & & 7.141e-01 & & 7.256e-02 & & 1.137e-01 & \\
		& $1/16$ & 1.881e-01 & 1.97 & 1.876e-01 & 1.93 & 1.852e-02 & 1.97 & 2.977e-02 & 1.93 \\
		& $1/32$ & 4.724e-02 & 1.99 & 4.752e-02 & 1.98 & 4.654e-03 & 1.99 & 7.535e-03 & 1.98\\
		& $1/64$ & 1.182e-02 & 2.00 & 1.192e-02 & 2.00 & 1.165e-03 & 2.00 & 1.890e-03 & 2.00\\
		& $1/128$ & 2.960e-03 & 2.00 & 2.985e-03 & 2.00 & 2.913e-04 & 2.00 & 4.729e-04 & 2.00\\
		\midrule
		\multirow{5}{*}{3} 
		& $1/4$ & 2.517e-01 & & 2.533e-01 & & 3.024e-02 & & 3.820e-02 & \\
		& $1/8$ & 1.827e-02 & 3.78 & 1.819e-02 & 3.80 & 2.404e-03 & 3.65 & 2.857e-03 & 3.74\\
		& $1/16$ & 1.193e-03 & 3.94 & 1.180e-03 & 3.95 & 1.618e-04 & 3.89 & 1.879e-04 & 3.93 \\
		& $1/32$ & 7.552e-05 & 3.98 & 7.448e-05 & 3.99 & 1.034e-05 & 3.97 & 1.192e-05 & 3.98 \\
		& $1/64$ & 4.739e-06 & 3.99 & 4.669e-06 & 4.00 & 6.509e-07 & 3.99 & 7.482e-07 & 3.99 \\
		\bottomrule
	\end{tabular}
\end{table}

\begin{table}[h]
	\centering
	\caption{The performance of the discrete method \eqref{FSP-MVEM} for Example \ref{example2} with $r=k=2$ and $m=3$.}
	\label{ex2-k2}
	\begin{tabular}{ccc|cc|cc|cc}
		\toprule
		\multirow{2}{*}{$h$} & \multicolumn{4}{c}{$\varepsilon=1$} & \multicolumn{4}{c}{$\varepsilon=10^{-1}$} \\
		\cmidrule(lr){2-5}\cmidrule(lr){6-9}
		& $ \rm Err_{\boldsymbol{\sigma}} $ & rate & $\rm Err_u$ & rate & $ \rm Err_{\boldsymbol{\sigma}} $ & rate & $\rm Err_u$ & rate\\
		\midrule
		$1/8$ & 5.020e-02 & & 3.641e-02 & & 8.236e-03 & & 6.950e-03 & \\
		$1/16$ & 6.541e-03 & 2.94 & 4.967e-03 & 2.87 & 1.138e-03 & 2.86 & 9.884e-04 & 2.81\\
		$1/32$ & 8.389e-04 & 2.96 & 6.603e-04 & 2.91 & 1.486e-04 & 2.94 & 1.326e-04 & 2.90\\
		$1/64$ & 1.061e-04 & 2.98 & 8.488e-05 & 2.96 & 1.892e-05 & 2.97 & 1.708e-05 & 2.96\\
		$1/128$ & 1.334e-05 & 2.99 & 1.074e-05 & 2.98 & 2.383e-06 & 2.99 & 2.162e-06 & 2.98\\
		\bottomrule
	\end{tabular}
\end{table}
\end{example}

\begin{example}\label{example3}
\normalfont
Next we verify the convergence of the discrete method \eqref{FSP-MVEM} with boundary layers.
The exact solution of the Poisson equation \eqref{poisson} is set to be $$\bar{u} = \sin(\pi x)\sin(\pi y).$$
We take the right-hand side term $f$ computed from \eqref{poisson} as the right-hand side function of problem \eqref{FSP0}. The explicit expression solution $u$ for problem \eqref{FSP0} with this right-hand term is unknown.
The solution $u$ possesses strong boundary layers when $\varepsilon$ is very small. Take $\varepsilon = 10^{-6}, 10^{-8}$, $10^{-10}$.
We measure the numerical error
\begin{equation*}
\textrm{Err}_3 = \varepsilon|Q_h^{\rm div}\nabla_h(\bar{u}-u_h)|_{1,h}+\|\nabla \bar{u}-Q_h^{\rm div}\nabla_h u_h\|_{0}
\end{equation*}
with $r=m=k$.
The numerical error $\textrm{Err}_3$ with different $\varepsilon$ , $h$ and $k$ is presented in Table~\ref{ex3-k123}, from which we can see that $\textrm{Err}_3\eqsim O(h^{k})$.
These convergence rates are robust and optimal, consistent with \eqref{errresult4}.

\begin{table}[ht]
	\centering
	\caption{$\rm Err_3$ of the discrete method \eqref{FSP-MVEM} for Example \ref{example3} with $r=m=k$.}
	\label{ex3-k123}
	\begin{tabular}{cccc|cc|cc}
		\toprule
		\multirow{2}{*}{$k$}& \multirow{2}{*}{$h$} & \multicolumn{6}{c}{$\varepsilon$} \\
		\cmidrule(lr){3-8}
		& & $10^{-6}$ & rate & $10^{-8}$ & rate & $10^{-10}$ & rate \\
		\midrule
		\multirow{5}{*}{1} 
		& $1/16$ & 1.624e-01 & & 1.624e-01 &  & 1.624e-01 & \\
		& $1/32$ & 8.125e-02 & 1.00 & 8.125e-02 & 1.00 & 8.125e-02 & 1.00 \\
		& $1/64$ & 4.064e-02 & 1.00 & 4.064e-02 & 1.00 & 4.064e-02 & 1.00 \\
		& $1/128$ & 2.032e-02 & 1.00 & 2.032e-02 & 1.00 & 2.032e-02 & 1.00 \\
		& $1/256$ & 1.016e-02 & 1.00 & 1.016e-02 & 1.00 & 1.016e-02 & 1.00 \\
		\midrule
		\multirow{5}{*}{2} 
		& $1/8$ & 4.780e-02 & & 4.780e-02 & & 4.780e-02 &\\
		& $1/16$ & 1.208e-02 & 1.98 & 1.208e-02 & 1.98 & 1.208e-02 & 1.98 \\
		& $1/32$ & 3.029e-03 & 2.00 & 3.029e-03 & 2.00 & 3.029e-03 & 2.00 \\
		& $1/64$ & 7.581e-04 & 2.00 & 7.580e-04 & 2.00 & 7.580e-04 & 2.00 \\
		& $1/128$ & 1.896e-04 & 2.00 & 1.896e-04 & 2.00 & 1.896e-04 & 2.00 \\
		\midrule
		\multirow{5}{*}{3} 
		& $1/4$ & 1.465e-02 & & 1.465e-02 & & 1.465e-02 & \\
		& $1/8$ & 1.882e-03 & 2.96 & 1.882e-03 & 2.96 & 1.882e-03 & 2.96\\
		& $1/16$ & 2.374e-04 & 2.99 & 2.374e-04 & 2.99 & 2.374e-04 & 2.99\\
		& $1/32$ & 2.977e-05 & 3.00 & 2.977e-05 & 3.00 & 2.977e-05 & 3.00\\
		& $1/64$ & 3.729e-06 & 3.00 & 3.726e-06 & 3.00 & 3.726e-06 & 3.00\\
		\bottomrule
	\end{tabular}
\end{table}
\end{example}
\begin{example}
\normalfont
In this example, we show the boundary layer phenomenon in the numerical solution using figures. We adopt the same right-hand side term $f$ as in Example \ref{example3}, and $h=\frac{1}{256}$. Although the explicit expression solution $u$ for problem \eqref{FSP0} with this right-hand term remains unknown, $\varepsilon^{-2}\boldsymbol{\sigma}_h$ provides an accurate approximation for $\nabla^2u$. To visually demonstrate the boundary layer phenomenon, we examine the spatial variation in the Frobenius norm of $\varepsilon^{-2}\boldsymbol{\sigma}_h$. Our observation focuses on the simplest case, where $r=0, k=m=1$. A series of figures is presented, showing the Frobenius norm of $\varepsilon^{-2}\boldsymbol{\sigma}_h$ with different $\varepsilon$. As seen in Fig. \ref{six_images}, the boundary layer phenomenon becomes apparent when $\varepsilon=10^{-2}$, and it becomes increasingly pronounced as $\varepsilon$ decreases.  
\begin{figure}[htbp]
\centering

% row1
\begin{subfigure}[b]{0.3\textwidth}
\centering
\includegraphics[width=\textwidth, trim={3.8cm 9cm 4cm 9cm}, clip]{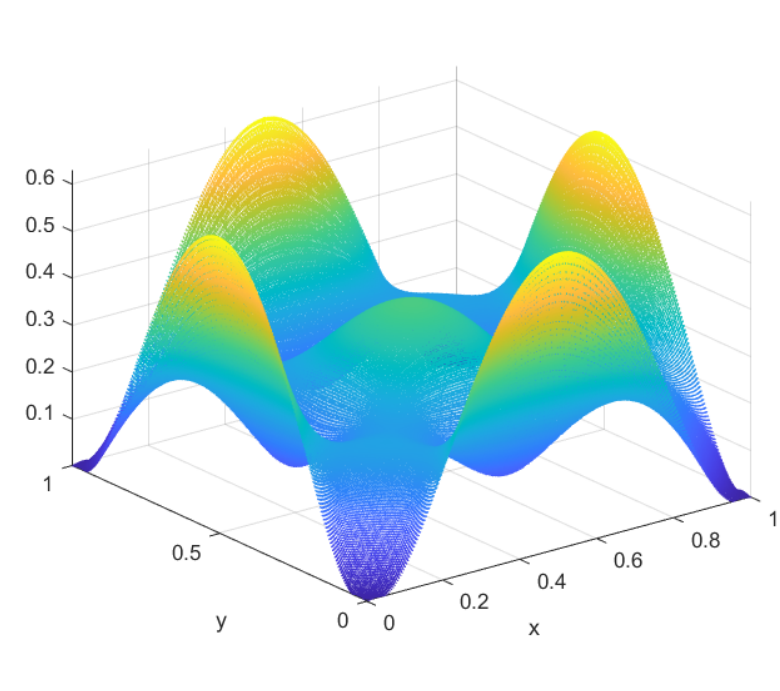}
\caption{$\varepsilon=1$}
\label{image1}
\end{subfigure}
\hfill
\begin{subfigure}[b]{0.3\textwidth}
\centering
\includegraphics[width=\textwidth, trim={3.8cm 9cm 4cm 9cm}, clip]{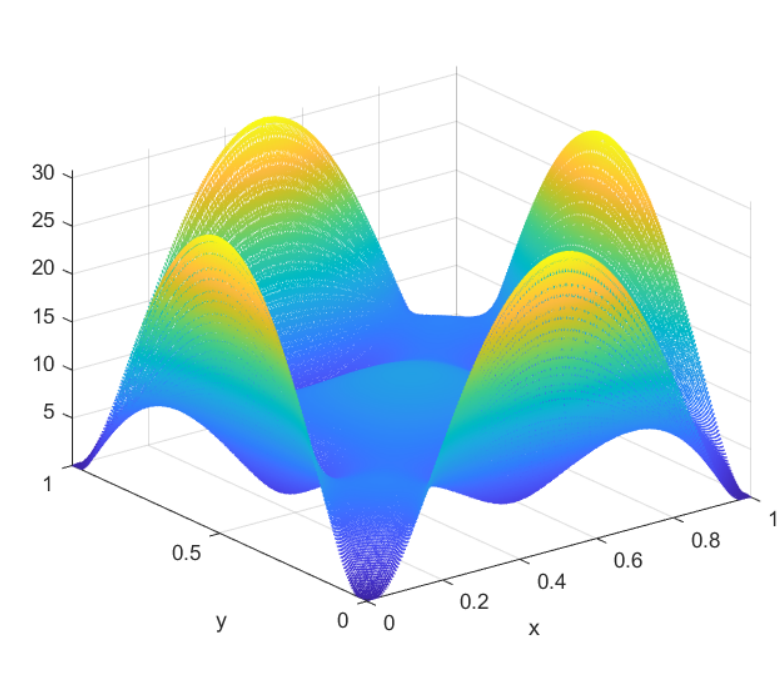}
\caption{$\varepsilon=10^{-1}$}
\label{image2}
\end{subfigure}
\hfill
\begin{subfigure}[b]{0.3\textwidth}
\centering
\includegraphics[width=\textwidth, trim={3.8cm 9cm 4cm 9cm}, clip]{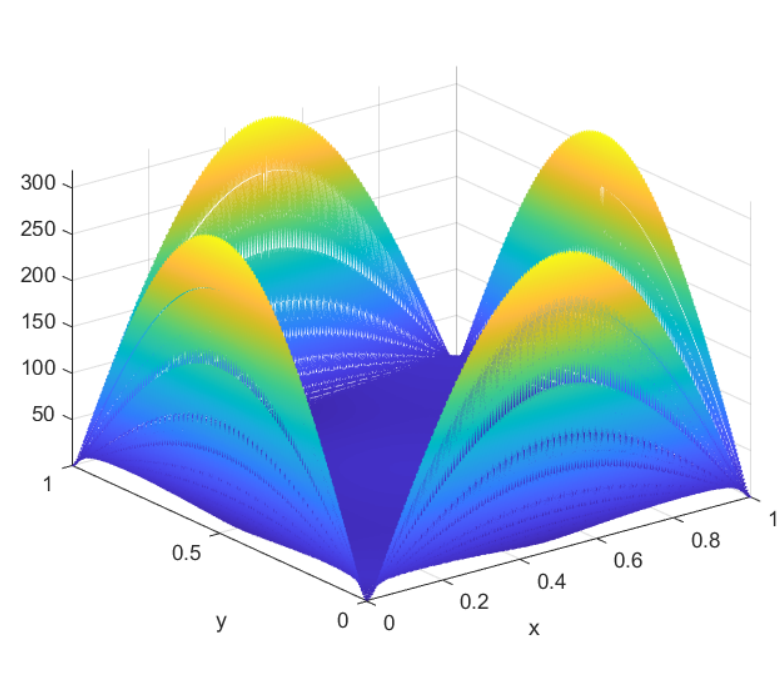} 
\caption{$\varepsilon=10^{-2}$}
\label{image3}
\end{subfigure}
%\vspace{1cm}
% row2
\begin{subfigure}[b]{0.3\textwidth}
\centering
\includegraphics[width=\textwidth, trim={3.8cm 9cm 4cm 9cm}, clip]{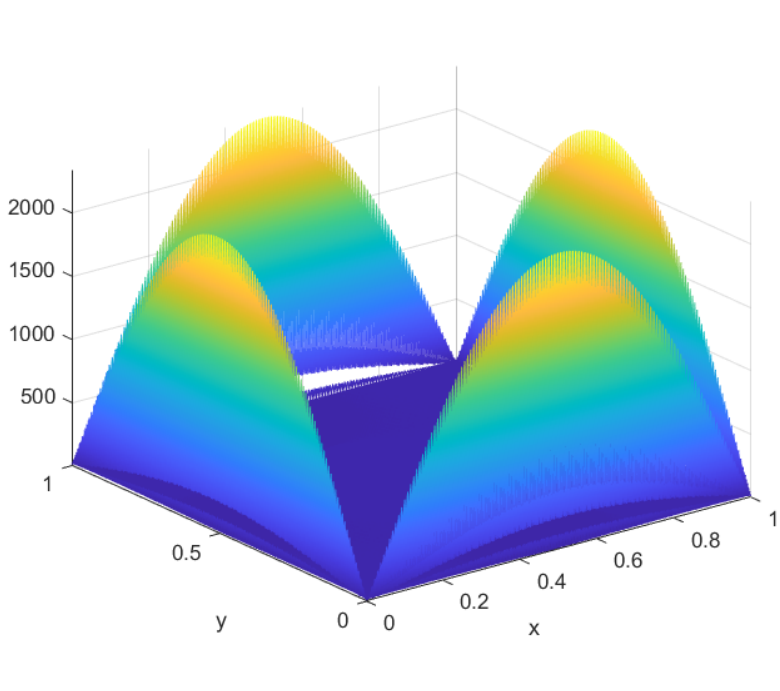}
\caption{$\varepsilon=10^{-3}$}
\label{image4}
\end{subfigure}
\hfill
\begin{subfigure}[b]{0.3\textwidth}
\centering
\includegraphics[width=\textwidth, trim={3.8cm 9cm 4cm 9cm}, clip]{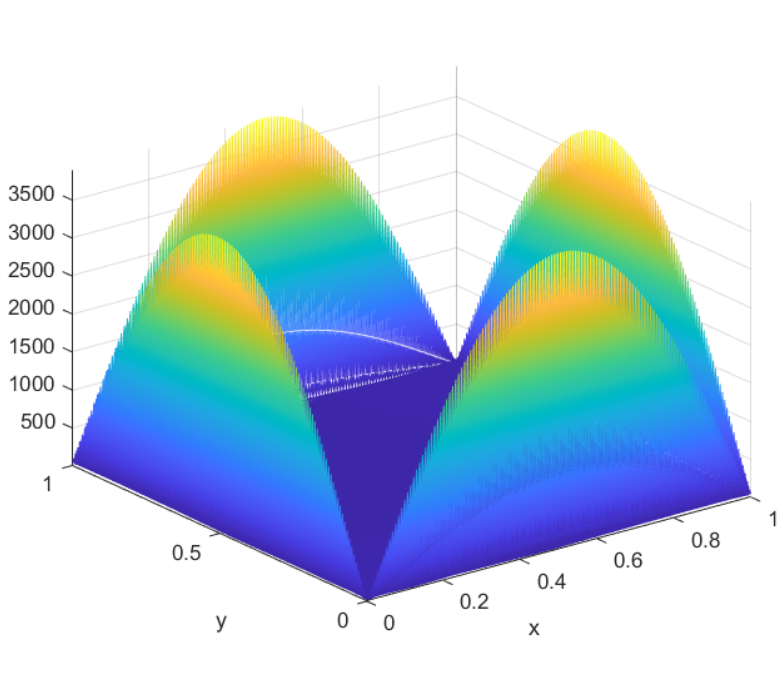}
\caption{$\varepsilon=10^{-4}$}
\label{image5}
\end{subfigure}
\hfill
\begin{subfigure}[b]{0.3\textwidth}
\centering
\includegraphics[width=\textwidth, trim={3.8cm 9cm 4cm 9cm}, clip]{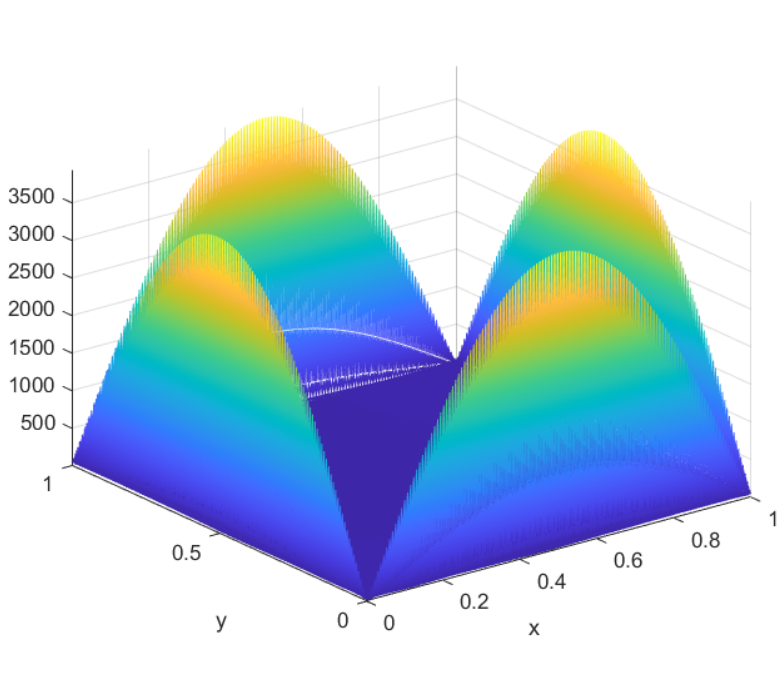}
\caption{$\varepsilon=10^{-5}$}
\label{image6}
\end{subfigure}

\caption{The Frobenius norm of $\varepsilon^{-2}\boldsymbol{\sigma}_h$ with different $\varepsilon$.}
\label{six_images}
\end{figure}
\end{example}

\bibliographystyle{abbrv} % F6+F8+F6+F8
\bibliography{ref} 
\end{document}